\documentclass[12pt]{amsart} 
\usepackage{hyperref}
\usepackage{fullpage}
\usepackage{amsmath,amsthm,amssymb}
\usepackage{subfigure}
\usepackage{graphicx}
\usepackage{mathtools}
\usepackage{dsfont}
\usepackage{enumitem}
\usepackage{tikz, color}
\usetikzlibrary{intersections,patterns,arrows,decorations.pathreplacing,shapes.misc}
\usepackage{algorithm}
\usepackage{algpseudocode}

 \newtheorem{theorem}{Theorem}[section]
 \newtheorem{corollary}[theorem]{Corollary}

 \newtheorem{proposition}[theorem]{Proposition}
 
 \theoremstyle{definition}
 \newtheorem{example}[theorem]{Example}

 \newtheorem{definition}[theorem]{Definition}
 
 \newtheorem{remark}[theorem]{Remark}

\newcommand{\Config}[2]{\mathrm{Config}_{#1, #2}}
\newcommand{\Stable}[2]{\mathrm{Stable}_{#1, #2}}

\newcommand{\DetSortRec}[2]{\mathrm{SortedDetRec}_{#1, #2}}

\newcommand{\Fer}[2]{\mathrm{Ferrers}_{#1, #2}}
\newcommand{\Para}[2]{\mathrm{ParaPoly}_{#1, #2}}
\newcommand{\DAGS}[2]{\mathrm{DAG}^{\mathrm{SR}}_{#1, #2}}
\newcommand{\DAGD}[2]{\mathrm{DAG}^{\mathrm{DR}}_{#1, #2}}
\newcommand{\Motz}[2]{\mathrm{LabMotz}_{#1, #2}}

\newcommand{\Configmn}{\mathrm{Config}_{m,n}}
\newcommand{\Stablemn}{\mathrm{Stable}_{m,n}}

\newcommand{\StoRecmn}{\mathrm{StoRec}_{m,n}}
\newcommand{\DetRecmn}{\mathrm{DetRec}_{m,n}}

\newcommand{\StoSortRecmn}{\mathrm{StoSortedRec}_{m,n}}
\newcommand{\DetSortRecmn}{\mathrm{DetSortedRec}_{m,n}}
\newcommand{\Fermn}{\mathrm{Ferrers}_{m,n}}
\newcommand{\DAGSmn}{\DAGS{m}{n}}
\newcommand{\DAGDmn}{\DAGD{m}{n}}

\newcommand{\inc}[1]{\mathrm{inc}\left(#1\right)}
\newcommand{\level}[1]{\mathrm{level}\left(#1\right)}

\newcommand{\DetStab}[1]{\mathrm{DetStab}\left(#1\right)}
\newcommand{\StoStab}[1]{\mathrm{StoStab}\left(#1\right)}

\newcommand{\R}{\mathbb{R}}

\newcommand{\Zp}{\mathbb{Z}_+}
\newcommand{\N}{\mathbb{N}}
\newcommand{\U}{\mathcal{U}}
\renewcommand{\L}{\mathcal{L}}

\newcommand{\sink}{\node (0) [draw, rectangle, fill=black] at (0,0) {};}

\newcommand{\ferrers}[2]{
  \begin{scope}[shift={#2}]
  \foreach[count=\yy] \xx in {#1}
    \draw [thick] (0,-\yy) grid (\xx,-\yy-1);
  \end{scope}
}

\newcommand{\Shift}{\mathsf{Shift}}
\newcommand{\Add}{\mathsf{Add}}

\newcommand{\tdot}[3]{\draw [fill=black,color=#3] (#1,#2) circle [radius=0.25];}

\definecolor{mygreen}{RGB}{20,180,20}

\begin{document}

\title{Abelian and stochastic sandpile models on complete bipartite graphs}
\author{Thomas Selig}
\address{Department of Computing, School of Advanced Technology, Xi'an Jiaotong-Livepool University}
\email{Thomas.Selig@xjtlu.edu.cn}
\author{Haoyue Zhu}
\email{Haoyue.Zhu18@student.xjtlu.edu.cn}

\begin{abstract}
In the sandpile model, vertices of a graph are allocated grains of sand. At each unit of time, a grain is added to a randomly chosen vertex. If that causes its number of grains to exceed its degree, that vertex is called unstable, and \emph{topples}. In the Abelian sandpile model (ASM), topplings are deterministic, whereas in the stochastic sandpile model (SSM) they are random. We study the ASM and SSM on complete bipartite graphs. For the SSM, we provide a stochastic version of Dhar's burning algorithm to check if a given (stable) configuration is recurrent or not, with linear complexity. We also exhibit a bijection between sorted recurrent configurations and pairs of \emph{compatible} Ferrers diagrams. We then provide a similar bijection for the ASM, and also interpret its recurrent configurations in terms of labelled Motzkin paths.
\end{abstract}

\maketitle

%%%%%%%%%%%%%%%%%%%%%%%% SECTION %%%%%%%%%%%%%%%%%%%%%

\section{Introduction}\label{sec:intro}

The Abelian sandpile model (ASM) is a dynamic process on a graph, where vertices are assigned a number of grains of sand. At each unit of time, a grain is added to a randomly chosen vertex. If this causes a vertex's number of grains to exceed its degree, the vertex is called \emph{unstable}, and \emph{topples}, sending one grain to each of its neighbours. A special vertex, the \emph{sink}, absorbs grains, and so the process eventually stabilises. The model was originally introduced by Bak, Tang and Wiesenfeld~\cite{BTW1,BTW2} as an example of a model exhibiting a phenomenon known as \emph{self-organised criticality}, before being formalised and generalised by Dhar~\cite{Dhar1}. 

Of central interest in the ASM are the \emph{recurrent configurations} -- those which appear infinitely often in the long-time running of the model. A fruitful direction of ASM research has focused on combinatorial studies of these for graph families with high levels of symmetry, such as complete graphs~\cite{CR}, complete bipartite~\cite{DLB} and multi-partite~\cite{CorPou} graphs, complete split graphs~\cite{DDLB,Duk} (see also~\cite{ADILBW} for generalisations), wheel and fan graphs~\cite{SelWheel}, Ferrers graphs~\cite{SSS,DSSS1}, permutation graphs~\cite{DSSS2}, and so on. Related combinatorial objects include parking functions and variations thereon, parallelogram polyominoes, lattice paths, subgraph structures, tableaux, and more.

In the ASM, the only randomness lies in the choice of vertex where grains are added at each time step. After this, the toppling and stabilisation processes are entirely deterministic. The main focus of this paper is instead a stochastic variant of the ASM, called \emph{stochastic sandpile model} (SSM), as introduced in~\cite{CMS}, in which topplings are made according to (biased) random coin flips. The SSM was studied on complete graphs in~\cite{SelSSM}. We begin by setting some notation and formally defining the model.

As usual, $\N$ denotes the set of strictly positive integers. We let $\Zp:=\N \, \cup \, \{0\}$ denote the set of non-negative integers. For $n \in \N$, we define $[n] := \{1,\ldots,n\}$. For a vector $a = (a_1, \ldots, a_n) \in \R^n$, we write $\inc{a} = \left( \tilde{a}_1,\ldots,\tilde{a}_n \right)$ for the non-decreasing rearrangement of $a$. In this paper, we consider the complete bipartite graph $K_{m,n}^0$. This is the graph with vertex set $\{v^t_0, v^t_1, \cdots, v^t_m\} \sqcup \{v^b_1, \cdots, v^b_n\}$ and edge set $\{ (v^t_i, v^b_j); \, i \in [m] \cup \{0\}, j \in [n] \}$. We refer to vertices $v^t_i$, resp.\ $v^b_j$, as \emph{top}, resp.\ \emph{bottom}, vertices in $K_{m,n}^0$. We will use the notation $v^*_i$ to refer to any arbitrary vertex of $K_{m,n}^0$. The vertex $v^t_0$, called the \emph{sink}, will play a special role in both the ASM and the SSM. Figure~\ref{fig:Detstab} shows the graph $K_{2,2}^0$, where the sink is represented as a black square. Finally, we fix a probability $p \in (0, 1)$.

A (sandpile) \emph{configuration} on $K_{m,n}^0$ is a vector $c = (c^t_1, \cdots, c^t_m; c^b_1, \cdots, c^b_n) \in \Zp ^ {m+n}$. For simplicity, we write $c = (c^t; c^b)$. We think of $c^*_i$ as the number of grains at vertex $v^*_i$. We denote by $\Configmn$ the set of all configurations on $K_{m,n}^0$. A top vertex $v^t_i$ (for $i \in [m]$), resp.\ bottom vertex $v^b_j$ (for $j \in [n]$), is \emph{stable} if $c^t_i < n$, resp.\ $c^b_j < m+1$ (i.e.\ the number of grains at the vertex is less than its degree). A configuration is stable if all of its vertices are stable. The set of stable configurations is denoted $\Stablemn$. 

Unstable vertices \emph{topple}. The toppling rules are different for the ASM and SSM, and we begin with the former. In the ASM, the toppling rule is deterministic: an unstable vertex (whether a top vertex $v^t_i$ or a bottom vertex $v^b_j$) will send one grain to each of its neighbours (this includes the sink $v^t_0$ for bottom vertices). This may cause other vertices who receive grains to become unstable, and there topple in turn. Grains that are sent to the sink exit the system (we think of the sink as ``absorbing excess grains'').

This last property makes it relatively straightforward to see that, starting from an unstable configuration $c$ and successively toppling unstable vertices, we eventually reach a stable configuration $c'$.
Moreover, Dhar~\cite{Dhar1} showed that the stable configuration $c'$ reached does not depend on the order in which vertices are toppled. Note that in the ASM this process is entirely deterministic.
We write $c' = \DetStab{c}$ and call it the \emph{deterministic stabilisation} of $c$. Figure~\ref{fig:Detstab} shows the deterministic stabilisation process for the configuration $c = (2, 1; 0, 2)$. Here, \textcolor{blue}{blue} vertices are unstable, \textcolor{mygreen}{green} edges represent grains being sent from an unstable vertex to a neighbour.

\begin{figure}[ht]
\centering
 
 \begin{tikzpicture}[scale=0.35]
 %original config
 \sink
 \node [draw, circle, color=blue] (1) at (3,0) {2};
 \node [draw, circle] (2) at (6,0) {1};
 \node [draw, circle] (3) at (1.5,-3.5) {0};
 \node [draw, circle] (4) at (4.5,-3.5) {2};
 \foreach \xx in {0,2}
   \foreach \yy in {3,4}
     \draw [thick] (\xx)--(\yy);
 \draw [thick, mygreen] (1)--(3);
 \draw [thick, mygreen] (1)--(4);
 \draw [->] (7.5,-1.2)--(9.5,-1.2);
 \node at (8.5,-2) {\textcolor{blue}{$v^t_1$}};
 
 %config1
 \begin{scope}[shift={(11,0)}]
 \sink
 \node [draw, circle] (1) at (3,0) {0};
 \node [draw, circle] (2) at (6,0) {1};
 \node [draw, circle] (3) at (1.5,-3.5) {1};
 \node [draw, circle, color=blue] (4) at (4.5,-3.5) {3};
 \foreach \xx in {0,1,2}
   \foreach \yy in {3}
     \draw [thick] (\xx)--(\yy);
 \draw [thick, mygreen] (4)--(0);
 \draw [thick, mygreen] (4)--(1);
 \draw [thick, mygreen] (4)--(2);
 \draw [->] (7.5,-1.2)--(9.5,-1.2);
 \node at (8.5,-2) {\textcolor{blue}{$v^b_2$}};
 \end{scope}
 
 %config2
 \begin{scope}[shift={(22,0)}]
 \sink
 \node [draw, circle] (1) at (3,0) {1};
 \node [draw, circle, color=blue] (2) at (6,0) {2};
 \node [draw, circle] (3) at (1.5,-3.5) {1};
 \node [draw, circle] (4) at (4.5,-3.5) {0};
 \foreach \xx in {0,1}
   \foreach \yy in {3,4}
     \draw [thick] (\xx)--(\yy);
 \draw [thick, mygreen] (2)--(3);
 \draw [thick, mygreen] (2)--(4);
 \draw [->] (7.5,-1.2)--(9.5,-1.2);
 \node at (8.5,-2) {\textcolor{blue}{$v^t_2$}};
 \end{scope}
 
 %stable config
 \begin{scope}[shift={(33,0)}]
 \sink
 \node [draw, circle] (1) at (3,0) {1};
 \node [draw, circle] (2) at (6,0) {0};
 \node [draw, circle] (3) at (1.5,-3.5) {2};
 \node [draw, circle] (4) at (4.5,-3.5) {1};
 \foreach \xx in {0,1,2}
   \foreach \yy in {3,4}
     \draw [thick] (\xx)--(\yy);
 \end{scope}
 
 \end{tikzpicture}
 
 \caption{Illustrating the deterministic stabilisation for $c = (2,1; 0,2) \in \Config{2}{2}$. Vertices under the arrows represent the vertex being toppled in that phase. \label{fig:Detstab}}

\end{figure}
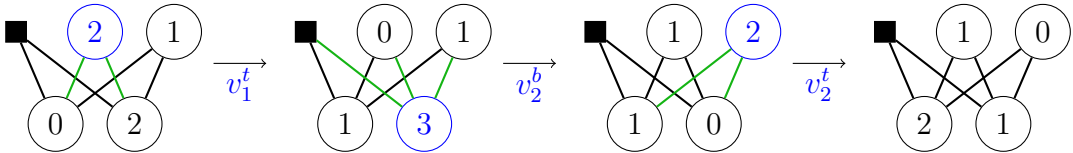

We now introduce the SSM. In this model, if a top vertex $v^t_i$ is unstable, then for each bottom neighbour $v^b_j$ we draw a Bernoulli random variable $B_j$ with parameter $p$ (the $B_j$'s are independent of each other and of all prior topplings). If $B_j = 1$, then $v^b_j$ receives one grain from $v^t_i$ when it topples, otherwise vertex $v^t_i$ keeps that grain. The process is the same for toppling a bottom vertex $v^b_j$ (as in the ASM we include the sink $v^t_0$ in the set of neighbours). The sink still acts as the system's exit point.

As in the ASM, one can show (see~\cite[Theorem 2.2]{CMS}) that, starting from an unstable configuration $c$ and successively toppling unstable vertices, we eventually reach a (random) stable stochastic configuration $c'$.
Moreover, the stochastic configuration $c'$ reached does not depend on the order in which vertices are toppled. We write $c' = \StoStab{c}$ and call it the \emph{stochastic stabilisation} of $c$. Figure~\ref{fig:Stostab} shows one possible example of the stochastic stabilisation process for the configuration $c = (2, 1; 0, 2)$. Here, \textcolor{blue}{blue} vertices are unstable, \textcolor{mygreen}{green} edges represent grains being sent from an unstable vertex to a neighbour, while \textcolor{red}{red} edges represent no movement of grain.

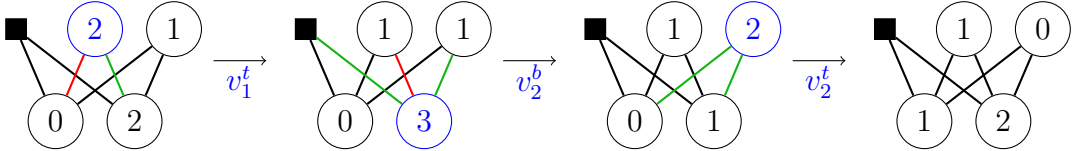
\begin{figure}[ht]
\centering
 
 \begin{tikzpicture}[scale=0.35]
 %original config
 \sink
 \node [draw, circle, color=blue] (1) at (3,0) {2};
 \node [draw, circle] (2) at (6,0) {1};
 \node [draw, circle] (3) at (1.5,-3.5) {0};
 \node [draw, circle] (4) at (4.5,-3.5) {2};
 \foreach \xx in {0,2}
   \foreach \yy in {3,4}
     \draw [thick] (\xx)--(\yy);
 \draw [thick, red] (1)--(3);
 \draw [thick, mygreen] (1)--(4);
 \draw [->] (7.5,-1.2)--(9.5,-1.2);
 \node at (8.5,-2) {\textcolor{blue}{$v^t_1$}};
 
 %config1
 \begin{scope}[shift={(11,0)}]
 \sink
 \node [draw, circle] (1) at (3,0) {1};
 \node [draw, circle] (2) at (6,0) {1};
 \node [draw, circle] (3) at (1.5,-3.5) {0};
 \node [draw, circle, color=blue] (4) at (4.5,-3.5) {3};
 \foreach \xx in {0,1,2}
   \foreach \yy in {3}
     \draw [thick] (\xx)--(\yy);
 \draw [thick, mygreen] (4)--(0);
 \draw [thick, red] (4)--(1);
 \draw [thick, mygreen] (4)--(2);
 \draw [->] (7.5,-1.2)--(9.5,-1.2);
 \node at (8.5,-2) {\textcolor{blue}{$v^b_2$}};
 \end{scope}
 
 %config2
 \begin{scope}[shift={(22,0)}]
 \sink
 \node [draw, circle] (1) at (3,0) {1};
 \node [draw, circle, color=blue] (2) at (6,0) {2};
 \node [draw, circle] (3) at (1.5,-3.5) {0};
 \node [draw, circle] (4) at (4.5,-3.5) {1};
 \foreach \xx in {0,1}
   \foreach \yy in {3,4}
     \draw [thick] (\xx)--(\yy);
 \draw [thick, mygreen] (2)--(3);
 \draw [thick, mygreen] (2)--(4);
 \draw [->] (7.5,-1.2)--(9.5,-1.2);
 \node at (8.5,-2) {\textcolor{blue}{$v^t_2$}};
 \end{scope}
 
 %stable config
 \begin{scope}[shift={(33,0)}]
 \sink
 \node [draw, circle] (1) at (3,0) {1};
 \node [draw, circle] (2) at (6,0) {0};
 \node [draw, circle] (3) at (1.5,-3.5) {1};
 \node [draw, circle] (4) at (4.5,-3.5) {2};
 \foreach \xx in {0,1,2}
   \foreach \yy in {3,4}
     \draw [thick] (\xx)--(\yy);
 \end{scope}
 
 \end{tikzpicture}
 
 \caption{Illustrating a possible stochastic stabilisation for $c = (2,1; 0,2) \in \Config{2}{2}$. Vertices under the arrows represent the vertex being toppled in that phase. \label{fig:Stostab}}

\end{figure}

For both ASM and SSM, we define a Markov chain on the set $\Stablemn$. At each step, we add a grain to a non-sink vertex of $K_{m,n}^0$, chosen uniformly at random, and stabilise the resulting configuration according to the chosen model's rules. A configuration $c$ is called \emph{deterministically recurrent} (DR), resp.\ \emph{stochastically recurrent} (SR) if it appears infinitely often in the long-time running of this Markov chain for the ASM, resp.\ SSM. We denote by $\DetRecmn$, resp.\ $\StoRecmn$, the set of DR, resp.\ SR, configurations on $K_{m,n}^0$. 

A popular statistic for a recurrent configuration $c$ is its \emph{level}, defined by:
\begin{equation}\label{eq:def_level}
\level{c} := \sum\limits_{i \in [m]} c^t_i + \sum\limits_{j \in [n]} c^b_j - m\!\cdot \! n,
\end{equation}
which satisfies $0 \leq \level{c} \leq m(n-1)$ (see e.g.~\cite[Equation~(8)]{SelSSM}).

The symmetries of $K_{m,n}^0$ make it natural to study both $\DetRecmn$ and $\StoRecmn$ up to re-ordering in each part. We therefore say that a configuration $c = (c^t; c^b) \in \Configmn$ is \emph{sorted} if $c^t$ and $c^b$ are both weakly increasing. We denote by $\DetSortRecmn$, resp.\ $\StoSortRecmn$, the set of sorted DR, resp.\ sorted SR, configurations for the ASM, resp.\ SSM.

Our paper is organised as follows. In Section~\ref{sec:SSM_compbipart} we provide a necessary and sufficient condition for a configuration on $K_{m,n}^0$ to be SR in terms of a certain sequence of inequalities (Theorem~\ref{thm:charac}). This then yields a so-called \emph{stochastic burning algorithm} for complete bipartite graphs that runs in linear time (Theorem~\ref{thm:burning}). In Section~\ref{sec:Ferrers} we provide combinatorial characterisations of SR configurations in terms of pairs of \emph{compatible Ferrers diagrams} (Theorem~\ref{thm:bij_rec_fer}). Section~\ref{sec:ASM_compbipart} examines equivalent of this new combinatorial description for the ASM, tying it to previous work by Dukes and Le Borgne~\cite{DLB}.

\section{Stochastically recurrent configurations on complete bipartite graphs}\label{sec:SSM_compbipart}

In this section, we give a characterisation of SR configurations on $K_{m,n}^0$. We use this to describe a \emph{stochastic burning algorithm}, which checks in linear time if a given configuration is SR or not.

\subsection{Characterisation of $\StoRecmn$}\label{subsec:charac}

In~\cite[Theorem~2.6]{SelSSM}, the first author gave a characterisation of SR configurations on general graphs. We begin by restating this in the complete bipartite graph case.

\begin{theorem}\label{thm:charac_forbidden_subconfig}
Let $c = (c^t;c^b) \in \Stablemn$ be a stable configuration on $K_{m,n}^0$. Then $c \in \StoRecmn$ if, and only if, for all subsets $A \subseteq [m], B \subseteq [n]$, we have:
\begin{equation}\label{eq:forbidden_subconfig}
\sum\limits_{i \in A} c^t_i + \sum\limits_{j \in B} c^b_j \geq \vert A \vert \cdot \vert B \vert.
\end{equation}
If $A,B$ do \emph{not} satisfy Inequality~\eqref{eq:forbidden_subconfig}, we say that $(A,B)$ is a \emph{forbidden subconfiguration}.
\end{theorem}

In fact, because of the structure of complete bipartite graph, it is not necessary to check Inequality~\eqref{eq:forbidden_subconfig} for all subsets $A,B$. Instead, intuitively, we only need to check it for a linear proportion of such subsets. This is made precise in the following, which is the main result of this section.

\begin{theorem}\label{thm:charac}
Let $c = (c^t; c^b) \in \Stablemn$ be a stable configuration on $K_{m,n}^0$.
For $j \in [n]$, define $k_j := \vert \{i \in [m]; \, c^t_i < j \} \vert$. Then $c \in \StoRecmn$ if, and only if,
\begin{equation}\label{eq:charac}
\forall j \in [n], \, \tilde{c}^b_1 + \cdots + \tilde{c}^b_j \geq k_1 + \cdots + k_j,
\end{equation}
where $\inc{c^b} := \left( \tilde{c}^b_1, \cdots, \tilde{c}^b_n \right)$ is the non-decreasing re-arrangement of $c^b$. Moreover, if $c$ is recurrent, we have $\level{c} = c^b_1 + \cdots + c^b_n - (k_1 + \cdots + k_n)$.
\end{theorem}

\begin{proof}
It is sufficient to show the result when $c$ is sorted, in which case, for $j \geq 0$, we have $c^t_i = j$ if, and only if, $k_j < i \leq k_{j+1}$ (with the convention $k_0 = 0$). Fix $j \in [n]$. We have:
\begin{equation}\label{eq:top-sum} 
\sum\limits_{i = 1}^{k_j} c^t_i = 0\!\cdot\!(k_1 - k_0) + 1\!\cdot\!(k_2 - k_1) + \cdots + (j-1)\!\cdot\!(k_j - k_{j-1}) = j\!\cdot\!k_j - (k_1 + \cdots + k_j).
\end{equation}
In particular, if we take $A = [k_j]$ and $B = [j]$, we have
\begin{align*}
\sum\limits_{i \in A} c^t_i + \sum\limits_{j' \in B} c^b_{j'} & = \sum\limits_{i = 1}^{k_j} c^t_i + \sum\limits_{j'
 = 1}^{j} c^b_{j'} \\
  & =  j\!\cdot\!k_j - (k_1 + \cdots + k_j) + (c^b_1 + \cdots + c^b_j) \\
  & =  \vert A \vert \cdot \vert B \vert + (c^b_1 + \cdots + c^b_j) - (k_1 + \cdots + k_j).
\end{align*} It follows that Inequality~\eqref{eq:charac} is equivalent to Inequality~\eqref{eq:forbidden_subconfig} when we take $A = [k_j], B = [j]$. It therefore suffices to show that if there exists a forbidden subconfiguration $(A,B)$ for $c$, then there exists $j$ such that $([k_j], [j])$ is a forbidden subconfiguration for $c$.

For this, take $B \subseteq [n]$ to be a minimal subset such that there exists $A \subseteq [m]$ with $(A,B)$ forbidden, that is, $\sum\limits_{i \in A} c^{t}_{i} + \sum\limits_{j \in B} c^b_{j} < \vert A \vert \cdot \vert B \vert$. Let $q := \max B$. We claim that for any $q' < q$, we have $c^b_{q'} < \vert A \vert$.
Otherwise, if $c^b_{q'} \geq \vert A \vert$, let $B' := [q'] \cap B \subsetneq B$. Since $c$ is assumed to be sorted, by using $c^b_j \geq c^b_{q'} \geq \vert A \vert$ for $j \in (q', q]$, we get $\sum\limits_{j \in B'}c^b_{j} = \sum\limits_{j \in B} c^b_{j} - \sum\limits_{j \in (q',q] \cap B} c^b_{j} \leq \sum\limits_{j \in B} c^b_{j} - \vert A \vert \cdot \vert B\setminus B' \vert $. Then we have:
\begin{align*}
\sum\limits_{i \in A}c^{t}_{i} + \sum\limits_{j \in B'} c^b_{j} & \leq \sum\limits_{i \in A} c^{t}_{i} + \sum\limits_{j \in B} c^b_{j} - \vert A \vert \cdot \vert B\setminus B' \vert \\
 & < \vert A \vert \cdot \vert B \vert - \vert A \vert \cdot \vert B \setminus B' \vert \\
 & = \vert A \vert \cdot \vert B' \vert,
\end{align*}
which implies that $(A,B')$ is a forbidden subconfiguration. But since $B' \subsetneq B$, this contradicts the minimality of $B$. This proves our claim that for any $q' < q$, we have $c^b_{q'} < \vert A \vert$, and in particular this holds for $q' \in [q] \setminus B$. 

We therefore get:
\begin{align*}
\sum\limits_{i \in A} c^{t}_{i} + \sum\limits_{j \in [q]} c^b_{j} & = \sum\limits_{i \in A} c^{t}_{i} + \sum\limits_{j \in B} c^b_{j} + \sum\limits_{j \in [q] \setminus B} c^b_{j} \\
 & < \sum\limits_{i \in A} c^{t}_{i} + \sum\limits_{j \in B} c^b_{j} + \vert A \vert \cdot \vert [q] \setminus B \vert \\
 & < \vert A \vert \cdot \vert B \vert + \vert A \vert \cdot \vert [q] \setminus B \vert \\
 & = \vert A \vert \cdot \vert [q] \vert,
\end{align*}
which implies that $(A, [q])$ is also forbidden, i.e., we can choose $B$ of the form $B = [q]$. By a similar argument (considering $A$ as a minimal subset such that $(A, [q])$ is forbidden), we can choose $A = [p]$.
From there, it is straightforward to check that $([k_q], [q])$ is also forbidden, by distinguishing the cases $p > k_q$ and $p < k_q$.
\begin{description}
\item[Case $p > k_q$] By definition of $k_q$, for any $i \in (k_q, p]$ we have $c^{t}_i \geq q = \vert [q] \vert$ (since $c^t$ is assumed sorted). This implies that
\begin{align*}
\sum\limits_{i \in [k_q]} c^t_{i} + \sum\limits_{j \in [q]} c^{b}_{j} & = \sum\limits_{i \in [p]} c^{t}_{i} + \sum\limits_{j \in [q]} c^b_{j} - \sum\limits_{i \in (k_q, p]} c^t_{i} \\
 & \leq \sum\limits_{i \in [p]} c^{t}_{i} + \sum\limits_{j \in [q]} c^b_{j} - (p - k_q)\!\cdot\!q \\
 & < p\!\cdot\!q - (p - k_q)\!\cdot\!q \\
 & = k_q\!\cdot\!q = \vert [k_q] \vert \cdot \vert [q] \vert,
\end{align*}
as desired.
\item[Case $p < k_q$] Again, by definition of $k_q$, for any $i \in (p, k_q]$, we have $c^{t}_i < \vert [q] \vert$, which implies: 
\begin{align*}
\sum\limits_{i \in [k_q]} c^t_{i} + \sum\limits_{j \in [q]} c^{b}_{j} & = \sum\limits_{i \in [p]} c^{t}_{i} + \sum\limits_{j \in [q]} c^b_{j} + \sum\limits_{i \in (p, k_q]} c^{t}_{i} \\
& < p\!\cdot\!q + (k_q - p)\!\cdot\!q \\
& = \vert [k_q] \vert \cdot \vert [q] \vert,
\end{align*}
as desired.
\end{description}
Finally, we have shown that if there exists a forbidden subconfiguration $(A,B)$ for $c$, then there exists $j$ such that $([k_j], j)$ is forbidden for $c$. This completes the proof of the characterisation of SR configurations.

The level formula then follows from the definition of the statistic in Equation~\eqref{eq:def_level}, and Equation~\eqref{eq:top-sum} with $j = n$, noting that $k_n = m$, which yields:
\begin{align*}
\level{c} & := \sum\limits_{i \in [m]} c^t_i + \sum\limits_{j \in [n]} c^b_j - m\!\cdot \! n \\
& = n\!\cdot\!m - (k_1 + \cdots + k_n) + \sum\limits_{j \in [n]} c^b_j - m\!\cdot \! n \\
& = c^b_1 + \cdots + c^b_n - (k_1 + \cdots + k_n),
\end{align*}
as desired.
\end{proof}

\begin{example}\label{ex:SR_config}
Consider the stable configuration $c = (c^t;c^b) = (3,1,3,2,3;2,0,4,3) \in \Stable{5}{4}$. We first sort the bottom part $c^b$ and get $c^b = \inc{c^b} = (0,2,3,4)$. Next, we calculate the vector $k$ as in the statement of Theorem~\ref{thm:charac}, yielding $k = (0, 1, 2, 5)$. We then check Condition~\eqref{eq:charac} for $j=1, 2, 3, 4$.
\begin{itemize}[topsep=2pt]
\item For $j = 1$, we have $\tilde{c^b_1} = 0 \geq 0 = k_1$.
\item For $j = 2$, we have $\tilde{c^b_1} + \tilde{c^b_2} = 0 + 2 = 2 \geq 1 = 0 + 1 = k_1 + k_2$.
\item For $j = 3$, we have $\tilde{c^b_1} + \tilde{c^b_2} +\tilde{c^b_3} = 5 \geq 3 = k_1 + k_2 + k_3$.
\item For $j = 4$, we have $\tilde{c^b_1} + \tilde{c^b_2} + \tilde{c^b_3} + \tilde{c^b_4} = 9 \geq 8 = k_1 + k_2 + k_3 + k_4$.
\end{itemize}
Therefore the configuration $c$ is SR by Theorem~\ref{thm:charac}, and we have $\level{c} = 9-8 = 1$.
\end{example}

\subsection{A stochastic burning algorithm for complete bipartite graphs}
\label{subsec:burning_SSM}

We now exhibit an algorithm to check if a given stable configuration $c \in \Stablemn$ is recurrent or not. The first step (Algorithm~\ref{algo:k}) is to calculate the vector $k = (k_1, \cdots, k_n)$ as defined in Theorem~\ref{thm:charac}. The second step (Algorithm~\ref{algo:burning}) then gives the desired check. The terminology \emph{burning algorithm} refers to Dhar's process for the deterministic ASM (see~\cite[Section~6.2]{Dhar}).

\begin{algorithm}
\caption{Pre-processing: calculating the vector $k = (k_1, \cdots, k_n)$}\label{algo:k}
  \begin{algorithmic}
      \Require $n \in \N$; $c^t = (c^t_1, \cdots, c^t_m) \in \{0, \cdots, n-1\}^m$
      \State \textbf{Initialise:} $k' = (k'_0, \cdots, k'_{n-1}) = (0, \cdots, 0)$; $k = (k_1, \cdots, k_n) = (0, \cdots, 0)$; $\mathrm{sum} = 0$
      \For {$i$ from $1$ to $m$} 
        \State $k'_{c^t_i} \gets k'_{c^t_i} + 1$ \Comment{Calculate $k'_j := \vert \{ i \in [m]; \, c^t_i = j \} \vert$}
      \EndFor
      \For {$j$ from $1$ to $n$}
        \State $\mathrm{sum} \gets \mathrm{sum} + k'_{j-1}$; $k_j \gets \mathrm{sum}$
      \EndFor \\
      \Return $k = (k_1, \cdots, k_n)$
  \end{algorithmic}
\end{algorithm}

\begin{proposition}\label{pro:algo_k}
Algorithm~\ref{algo:k} outputs the vector $k = (k_1, \cdots, k_n)$ defined by $k_j = \vert \{i \in [m]; \, c^t_i < j \} \vert$ for any $j \in [n]$. Moreover, the algorithm runs in $O(m+n)$ time.
\end{proposition}

\begin{proof}
The first loop calculates the vector $k' = (k'_0, \cdots, k'_{n-1})$ defined by $k'_j = \vert \{i \in [m]; \, c^t_i = j \} \vert$ for $0 \leq j \leq n-1$, in $O(m)$ time. The second loop then calculates the vector $k$ defined by $k_j = k'_0 + \cdots + k'_{j-1}$ for $j \in [n]$, from which the formula immediately follows, in $O(n)$ time.
\end{proof}

\begin{algorithm}
\caption{Stochastic burning algorithm for complete bipartite graphs}\label{algo:burning}
  \begin{algorithmic}
    \Require $c = (c^t; c^b) \in \Stablemn$
    \State $c^b \gets \inc{c^b}$ \Comment{Sort bottom part $c^b$ of configuration $c$}
    \State \textbf{Pre-process:} calculate vector $k = (k_1, \cdots, k_n)$ by Algorithm~\ref{algo:k}
    \State \textbf{Initialise:} $\mathrm{sumK} = 0$; $\mathrm{sumC} = 0$
    \For {$j$ from $1$ to $n$}
      \State $\mathrm{sumK} \gets \mathrm{sumK} + k_j$; $\mathrm{sumC} \gets \mathrm{sumC} + c^b_j$
      \If{$\mathrm{sumC} < \mathrm{sumK}$}
         \State \textbf{return} False
      \EndIf
    \EndFor \\
    \Return True
  \end{algorithmic}
\end{algorithm}

\begin{theorem}\label{thm:burning}
Algorithm~\ref{algo:burning} returns True if, and only if, the input (stable) configuration $c$ is stochastically recurrent. Moreover, the algorithm runs in $O(m + n)$ time.
\end{theorem}

\begin{proof}
The recurrence check follows from Theorem~\ref{thm:charac}. Algorithm~\ref{algo:k} calculates the vector $k$ in $O(m + n)$ time. Now note that for all $j \in [n]$, we have $0 \leq c^b_j \leq m$ (since $c$ is stable). As such, sorting $c^b$ can be done in $O(m + n)$ time by using the $\mathrm{CountSort}$ algorithm (also known as \emph{sort by values}, see e.g.~\cite[proof of Proposition~13]{CLB_Baker}). In brief, this algorithm first computes an auxiliary array $a = (a_0, \dots, a_{m})$ where $a_j = \left\vert \{i \in [n]; c^b_i = j \} \right\vert$ for all $0 \leq j \leq m$. This is analogous to the calculation of the vector $k'$ in the first step of Algorithm~\ref{algo:k}, and thus takes $O(n)$ time. We then recover $\inc{c^b}$ from $a$ in linear time by looping over $a$ and appending $a_j$ times the value $j$ to $\inc{c^b}$ at each step (since $\sum\limits_{j=0}^m a_j = n$ this takes $O(m+n)$ time). The rest of Algorithm~\ref{algo:burning} then runs in $O(n)$ time.
\end{proof}

\section{Recurrent configurations as pairs of Ferrers diagrams}
\label{sec:Ferrers}

In this section, we present a combinatorial interpretation of $\StoSortRecmn$ in terms of \emph{Ferrers diagrams}. 
A Ferrers diagram is a left-aligned collection of cells such that the number of cells in each row is weakly increasing from bottom to top (some rows may be empty). We denote by $\Fermn$, resp.\ $\Fer{\leq m}{n}$, the set of Ferrers diagrams with $m$ columns, resp.\ at most $m$ columns, and $n$ rows. The \emph{area} $\mathrm{Area}(F)$ of a Ferrers diagram $F$ is its number of cells. Given a weakly increasing sequence $s = (s_1, \cdots, s_n) \in \Zp^n$, we denote $F(s) \in \Fer{s_n}{n}$ the Ferrers diagram with $s_i$ cells in row $i$ (rows are ordered from bottom to top). For example, $F(0, 1, 4) := $ \begin{tikzpicture} [scale=0.2]
\ferrers{4,1,0}{(0,0)} 
\end{tikzpicture}
is an element of $\Fer{4}{3}$ with area $5$.

We consider the following two families of operations on Ferrers diagrams:
\begin{enumerate}
\item $\Shift$ which shifts a cell of the diagram in a given row to some row below.
\item $\Add$ which adds a cell to the right of a given row.
\end{enumerate}
A $\Shift$ or $\Add$ operation is called \emph{legal} if it still results in a Ferrers diagram (possibly with a different number of columns). Figure~\ref{fig:comp} illustrates these operations.

\begin{definition}\label{def:comp}
We say that an ordered pair $(F, F')$ of Ferrers diagrams is \emph{compatible} if $F'$ can be obtained from $F$ through a sequence of legal $\Shift$ and $\Add$ operations.
\end{definition}

We now return to our study of $\StoRecmn$. Note that the vectors $k := (k_1, \cdots, k_n)$ and $\inc{c^b} = \left( \tilde{c}^{b}_1, \cdots, \tilde{c}^{b}_n \right)$ which appear in Inequality~\eqref{eq:charac} are both weakly increasing. Moreover, we have $k_n = m$ and $\tilde{c}^b_n \leq m$ (the latter is the stability condition). This yields the following.

\begin{theorem}\label{thm:bij_rec_fer}
For a \emph{sorted} configuration $c = (c^t; c^b) \in \Configmn$, we define $\Psi(c) := \Big(F(k), F\left( c^b \right) \Big) \in \Fermn \times \Fer{\leq m}{n}$, where $k = (k_1, \cdots, k_n)$ is as in the statement of Theorem~\ref{thm:charac}. Then $\Psi$ is a bijection from $\StoSortRecmn$ to the set of compatible pairs $(F, F') \in \Fermn \times \Fer{\leq m}{n}$. Moreover, we have $\level{c} = \mathrm{Area}\big( F\left( c^b \right) \big) - \mathrm{Area} \left( F(k) \right)$, which is also the number of $\Add$ operations in a legal sequence $F(k) \leadsto F(c^b)$.
\end{theorem}

\begin{proof}
By preceding remarks, if $c \in \StoRecmn$, then $F(k) \in \Fermn$ and $F\left(c^b\right) \in \Fer{\leq m}{n}$. To show that $\Psi$ is injective, we note that the vector $k$ uniquely defines the non-decreasing re-arrangement $\inc{c^t}$ of the top part of $c$. Indeed, let $k'_j := \vert \{ i \in [m]; \, c^t_i = j \} \vert$ for $0 \leq j \leq n-1$. Then we have $k'_j = k_{j+1} - k_j$ (with the convention $k_0 = 0$), and the vector $\inc{c^t}$ simply consists of the values $j$ repeated $k'_j$ times for $0 \leq j \leq n-1$. Moreover, the level formula follows from the level formula in Theorem~\ref{thm:charac}, which gives $\level{c} = c^b_1 + \cdots + c^b_n - (k_1 + \cdots + k_n) = \tilde{c}^b_1 + \cdots + \tilde{c}^b_n - (k_1 + \cdots + k_n) = \mathrm{Area}\big( F\left( c^b \right) \big) - \mathrm{Area} \left( F(k) \right)$. Since $\Shift$ operations leave the area of a Ferrers diagram unchanged, and $\Add$ operations increase it by one, this is also the number of $\Add$ operations in a legal sequence $F(k) \leadsto F(c^b)$.

It remains to show that $\Psi$ is surjective, i.e.\ that $c$ is SR if, and only if, $\Big(F(k), F\left( c^b \right) \Big)$ is a compatible pair of Ferrers diagrams. Suppose first that $\Big(F(k), F\left( c^b \right) \Big)$ is compatible. This means that $F\left( c^b \right)$ can be reached from $F(k)$ through a sequence of legal $\Shift$ and $\Add$ operations. As such, by Theorem~\ref{thm:charac}, it is sufficient to show the following. If $v = (v_1, \cdots, v_n)$ is such that $v_1 + \cdots + v_j \geq k_1 + \cdots + k_j$ for all $j \in [n]$, and $v' = (v'_1, \cdots, v'_n)$ is such that $F(v')$ is obtained from $F(v)$ through a single legal $\Shift$ or $\Add$ operation, then we have $v'_1 + \cdots + v'_j \geq k_1 + \cdots + k_j$ for all $j \in [n]$. In other words, Condition~\eqref{eq:charac} is preserved through applying a single legal $\Shift$ or $\Add$ operation to the Ferrers diagram of the left-hand side vector. Adding a cell to row $p$ in $F(v)$ simply increases $v_p$ by one, leaving other values unchanged, so the inequalities are clearly preserved. Shifting a cell from row $p$ to row $p' < p$ yields $v'_p = v_p - 1$, $v'_{p'} = v_{p'} + 1$, and other values unchanged. Thus the sum $v_1 + \cdots + v_j$ is increased by one if $j \in [p', p)$ and unchanged otherwise, so that Condition~\eqref{eq:charac} is preserved.

We now show the converse, namely that if $c$ is SR, then the pair $\Big(F(k), F\left( c^b \right) \Big)$ is compatible. Let $c$ be a \emph{sorted} SR configuration, and $k$ the corresponding vector. First, note that if $c^b_j \geq k_j$ for all $j \in [n]$, then $F\left( c^b \right)$ can be obtained from $F(k)$ by adding the required cells in each row from top to bottom, which is a legal sequence, and the result follows. We therefore assume this is not the case, and define $p := \min \{ j \in [n]; \, c^b_j < k_j \}$. Since $c$ is SR, we have $c^b_1 + \cdots c^b_p \geq k_1 + \cdots + k_p$ by Theorem~\ref{thm:charac}. 
Therefore there must exist $q \in [p]$ such that $c^b_q > k_q$, and we set $p'$ to be the maximal such $q$. Now note that on the one hand, we have $k_p > c^b_p \geq c^b_{p-1} \geq k_{p-1}$, i.e.\ $k_p > k_{p-1}$, while on the other hand, we have $k_{p'} < c^b_{p'} \leq c^b_{p'+1} \leq k_{p'+1}$, i.e.\ $k_{p'} < k_{p'+1}$. These two conditions combined mean that shifting a cell from row $p$ to row $p'$ in $F(k)$ is a legal operation. 
But this shift yields a Ferrers diagram whose vector $v$ satisfies $c^b_j \geq v_j$ for all $j < p$, and $v_p - c^b_p < k_p - c^b_p$. Iterating this construction will eventually allow us to reach a Ferrers diagram whose vector $v$ satisfies $v_p - c^b_p = 0$. We then iterate on the index $p$ to eventually reach a Ferrers diagram whose vector $v$ satisfies $c^b_j \geq v^b_j$ for all $j \in [m]$, which brings us back to the previous case, from where a legal series of $\Add$ additions can be performed to reach $F\left( c^b \right)$, as desired.
This completes the proof.
\end{proof}

\begin{example}\label{ex:rec_fer_comp}
Consider the sorted configuration $c = (0, 2, 2; 2, 2, 2)$. We have $k = (1, 1, 3)$, so $c$ is recurrent by Theorem~\ref{thm:charac}. Figure~\ref{fig:comp} illustrates a possible legal sequence of $\Shift$ and $\Add$ operations to go from the $k$-diagram (left) to the $c^b$-diagram (right). Note that this sequence is not unique: we could instead first add a cell in the middle row, then shift a cell from the top to the bottom row. Indeed, our proof of Theorem~\ref{thm:bij_rec_fer} implies that the sequence can always be chosen to first perform only $\Shift$ operations, followed by the $\Add$ operations.

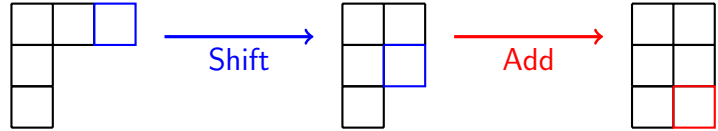
\begin{figure}[ht]
\centering
\begin{tikzpicture}[scale = 0.55]
\ferrers{2,1,1}{(0,0)}
\draw [blue, thick] (2,-1)--(3,-1)--(3,-2)--(2,-2)--(2,-1);
\draw [blue,very thick,->] (3.7,-1.8)--(7.3,-1.8);
\node [below] at (5.5,-1.8) {\textcolor{blue}{$\Shift$}};
\ferrers{2,1,1}{(8,0)}
\draw [blue, thick] (9,-2)--(10,-2)--(10,-3)--(9,-3)--(9,-2);
\draw [red,very thick,->] (10.7,-1.8)--(14.3,-1.8);
\node [below] at (12.5,-1.8) {\textcolor{red}{$\Add$}};
\ferrers{2,2,1}{(15,0)}
\draw [red, thick] (16,-3)--(17,-3)--(17,-4)--(16,-4)--(16,-3);
\end{tikzpicture}
\caption{Illustrating a legal sequence from $F(1,1,3)$ to $F(2,2,2)$, showing that the configuration $c = (0,2,2;2,2,2)$ is stochastically recurrent. \label{fig:comp}}
\end{figure}
\end{example}

\begin{remark}\label{rem:k-diag_to_config}
The top configuration $\inc{c^t}$ can be recovered from the Ferrers diagram $F(k)$ by taking the heights of East steps along the South-East border of the diagram, from the bottom-left corner to the top-right. For example, for the diagram $F(1,1,3)$ (Figure~\ref{fig:comp}, left), the South-East border can be written as $ENNEEN$, with $E$ denoting an East step, and $N$ a North one. Then we have $c^t = (0, 2, 2)$, corresponding to the heights of the $E$ steps.
\end{remark}

\begin{remark}\label{rem:all_config}
Theorem~\ref{thm:bij_rec_fer} gives a bijective representation of \emph{sorted} recurrent configurations. For the unsorted case, we label the Ferrers diagrams as follows. For the $k$-diagram $F(k)$, we assign bijectively to each column an element of $[m]$ such that columns of the same height are labelled in increasing order from left to right. Similarly, we label the rows of the $c^b$-diagram with elements of $[n]$ such that rows of the same length are labelled in increasing order from bottom to top. This yields a bijection from all recurrent configurations to the set of compatible labelled Ferrers diagrams.
\end{remark}

Finally, we propose a representation of the compatibility notion through a directed acyclic graph (DAG). The vertices of the graph are the Ferrers diagrams $F \in \Fer{\leq m}{n}$ satisfying $\mathrm{Area}(F) \geq m$. For every pair $(F,F')$, we put an edge from $F$ to $F'$ if $F' = \Shift(F)$ or $F' = \Add(F)$. As in Figure~\ref{fig:comp}, $\Shift$ edges are coloured \textcolor{blue}{blue}, and $\Add$ edges are \textcolor{red}{red}.
We denote $\DAGSmn$ the DAG thus obtained. Note that $\DAGSmn$ is \emph{bipolar}: it has a unique source $F(0, \cdots, 0, m)$ and a unique sink $F(m, \cdots, m)$. With this representation, $\Psi$ is a bijection from $\StoSortRecmn$ to pairs of vertices $(F,F')$ of $\DAGSmn$ such that $F$ has $m$ columns and there is a (directed) path from $F$ to $F'$ in $\DAGSmn$. The level of the configuration equals the number of \textcolor{red}{red} edges in such a path. Figure~\ref{fig:dag} illustrates this construction.

\begin{figure}[ht]
\centering
\begin{tikzpicture}[scale=0.35]
 %level 0
 \ferrers{3,0,0}{(0,0)}
 \draw [thick] (-0.1,-3)--(0.1,-3);
 \draw [blue,very thick,->] (1.2,-3)--(1.2,-5.1);
 \ferrers{2,1,0}{(0,-4.5)}
 \draw [blue,very thick,->] (0.8,-8.2)--(0.8,-9.6);
 \ferrers{1,1,1}{(0,-9)}
 %level 0 to level 1
 \draw [red,very thick,->] (2,-3)--(5.2,-2.2);
 \draw [red,very thick,->] (2.5,-6)--(5.1,-2.7);
 \draw [red,very thick,->] (2,-7)--(5,-7);
 \draw [red,very thick,->] (2,-7.5)--(5.1,-11);
 \draw [red,very thick,->] (1.5,-11.5)--(5.2,-12);
 %level 1
 \ferrers{3,1,0}{(6,0.5)}
 \draw [blue,very thick,->] (7.3,-3)--(7.3,-5.1);
 \ferrers{2,2,0}{(6,-4.5)}
 \draw [blue,very thick,->] (6.9,-8.2)--(6.9,-10.1);
 \ferrers{2,1,1}{(6,-9.5)}
 \draw [blue,very thick,out=-95,in=90] (5.6,-2) to (5.4,-7);
 \draw [blue,very thick,out=-90,in=95,->] (5.4,-7) to (5.6,-12);
 %level 1 to level 2
 \draw [red,very thick,->] (8,-2.5)--(12.1,-2);
 \draw [red,very thick,->] (8.5,-3)--(11.8,-6.7);
 \draw [red,very thick,->] (8.5,-7)--(11.8,-2.8);
 \draw [red,very thick,->] (8.5,-7.5)--(12.2,-11.4);
 \draw [red,very thick,->] (7.5,-12.2)--(12.5,-12.2);
 \draw [red,very thick,->] (8.5,-11.5)--(11.8,-7.3);
 \begin{scope}[shift={(0.8,0)}]
 %level 2
 \ferrers{3,2,0}{(12,0.5)}
 \draw [blue,very thick,->] (13.3,-3)--(13.3,-5.1);
 \ferrers{3,1,1}{(12,-4.5)}
 \draw [blue,very thick,->] (13.8,-7.5)--(13.8,-10.1);
 \ferrers{2,2,1}{(12.5,-9.5)}
 \draw [blue,very thick,out=-95,in=90] (11.6,-2) to (11.4,-7);
 \draw [blue,very thick,out=-90,in=105,->] (11.4,-7) to (12.1,-12);
 %level 2 to level 3
 \draw [red,very thick,->] (14.5,-2.5)--(17.6,-6.5);
 \draw [red,very thick,->] (14.8,-2)--(17.6,-2);
 \draw [red,very thick,->] (14.5,-7)--(17.6,-7);
 \draw [red,very thick,->] (15,-11.5)--(17.6,-7.5);
 \draw [red,very thick,->] (15,-12)--(18.1,-12);
 %level 3
 \ferrers{3,3,0}{(18,0.5)}
 \draw [blue,very thick,->] (19.8,-3)--(19.8,-5.1);
 \ferrers{3,2,1}{(18,-4.5)}
 \draw [blue,very thick,->] (19.6,-8.2)--(19.6,-10.1);
 \ferrers{2,2,2}{(18.5,-9.5)}
 %level 3 to level 4
 \draw [red,very thick,->] (21.5,-2.5)--(23.6,-4);
 \draw [red,very thick,->] (21.5,-6.5)--(23.6,-5);
 \draw [red,very thick,->] (21.2,-7)--(23.6,-9);
 \draw [red,very thick,->] (21,-11.5)--(23.6,-10);
 %level 4
 \ferrers{3,3,1}{(24,-2)}
 \draw [blue,very thick,->] (25.6,-5.5)--(25.6,-7.6);
 \ferrers{3,2,2}{(24,-7)}
 %level 4 to level 5
 \draw [red,very thick,->] (27.5,-5)--(29.4,-6.5);
 \draw [red,very thick,->] (27.5,-9)--(29.4,-7.5);
 %level 5
 \ferrers{3,3,2}{(29.8,-4.5)}
 %level 5 to level 6
 \draw [red,very thick,->] (33.3,-7)--(35,-7);
 %level 6
 \ferrers{3,3,3}{(35.4,-4.5)}
 \end{scope}
\end{tikzpicture}
\caption{The graph $\DAGS{3}{3}$. \label{fig:dag}}
\end{figure}

\section{Deterministically recurrent configurations on complete bipartite graphs}\label{sec:ASM_compbipart}

In this section, we return to the ASM, and its DR configurations. The ASM on complete bipartite graphs has already been extensively studied in~\cite{AADB,AADHB,DLB}. Our goal here is to describe DR configurations in terms of Ferrers diagrams, as we did for the SSM in Section~\ref{sec:Ferrers}.

\subsection{Characterisation of $\DetRecmn$}\label{subsec:ASM_charac}

We begin with a characterisation of DR configurations in the spirit of Theorem~\ref{thm:charac}.

\begin{theorem}\label{thm:charac_ASM}
Let $c = (c^t; c^b) \in \Stablemn$ be a stable configuration on $K_{m,n}^0$.
As previously, for $j \in [n]$, define $k_j := \vert \{i \in [m]; \, c^t_i < j \} \vert$. Then $c \in \DetRecmn$ if, and only if,
\begin{equation}\label{eq:charac_ASM}
\forall j \in [n], \, \tilde{c}^b_j \geq k_j,
\end{equation}
where $\tilde{c}^b_j$ is the $j$-th item from the non-decreasing re-arrangement of $c^b$.
\end{theorem}

\begin{proof}[Proof sketch]
It is sufficient to show the result for sorted configurations. Suppose that $c \in \Stablemn$ is sorted, and that there exists $j \in [n]$ such that $c^b_j < k_j$. Then we have $c^b_{j'} < k_j$ for all $j' \in [j]$. Moreover, by definition, we have $c^t_i < j$ for all $i \in [k_j]$. It follows that the configuration $c$ restricted to the induced subgraph on $([k_j], [j])$ is stable. This is exactly the definition of a forbidden subconfiguration for the ASM (see e.g.~\cite[Section~6]{Dhar}), and therefore $c$ is not recurrent. The converse can be shown in analogous manner to Theorem~\ref{thm:charac}: if there is a forbidden subconfiguration (in the above sense) $(A,B)$, we show that we can choose $A = [k_j]$ and $B = [j]$ for some $j \in [n]$, and the stability condition for $c$ then implies that $c^b_j < k_j$, as desired.
\end{proof}

Theorem~\ref{thm:charac_ASM} immediately implies the following.

\begin{corollary}
There exists an algorithm that checks in linear $O(m+n)$ time if a configuration $c \in \Configmn$ is recurrent or not.
\end{corollary}

\subsection{Ferrers diagram interpretation}\label{subsec:Ferrers_ASM}

As in Section~\ref{sec:Ferrers}, Theorem~\ref{thm:charac_ASM} allows for a characterisation of DR configurations in terms of pairs of Ferrers diagrams. Note that this time we must have $\tilde{c}^b_n = m$ since $\tilde{c}^b_n \geq k_n$ by Equation~\eqref{eq:charac_ASM}, and $k_n = m$. Moreover, Equation~\eqref{eq:charac_ASM} means exactly that each row of $F' := F\Big( \inc{c^b} \Big)$ has a number of cells that is greater than or equal to the number of cells in the same row of $F := F(k)$.

\begin{definition}\label{def:comp_ASM}
We say that an ordered pair $(F, F')$ of Ferrers diagrams is \emph{strongly compatible} if $F'$ can be obtained from $F$ through a sequence of legal $\Add$ operations for ASM case.
\end{definition}

We then get the following characterisation of $\DetSortRecmn$ in terms of pairs of Ferrers diagrams. The proof is analogous to that of Theorem~\ref{thm:bij_rec_fer}, combined with the straightforward observation that Equation~\eqref{eq:charac_ASM} is equivalent to $\Big(F(k), F\left( c^b \right) \Big) $ being strongly compatible.

\begin{theorem}\label{thm:bij_rec_fer_ASM}
For a \emph{sorted} configuration $c = (c^t; c^b) \in \Configmn$, we define $\Psi(c) := \Big(F(k), F\left( c^b \right) \Big) \in \Fermn \times \Fermn$, where $k = (k_1, \cdots, k_n)$ is as in Theorem~\ref{thm:charac_ASM}. Then $\Psi$ is a bijection from $\DetSortRecmn$ to the set of strongly compatible pairs $(F, F') \in \Fermn \times \Fermn$. Moreover, we have $\level{c} = \mathrm{Area}\big( F\left( c^b \right) \big) - \mathrm{Area} \left( F(k) \right)$, which is also the number of $\Add$ operations in a legal sequence $F(k) \leadsto F(c^b)$.
\end{theorem}

\begin{example}\label{ex:rec_fer_comp_ASM}
Consider the sorted configuration $c = (0, 2, 2; 2, 2, 3)$. We have $k = (1, 1, 3)$, $\tilde{c}^b = (2, 2, 3)$, and for all $j \in \{1,2,3\}$, $\tilde{c}^b_j \geq k_j$, so $c$ is DR by Theorem~\ref{thm:charac_ASM}. 
Figure~\ref{fig:comp_ASM} shows a possible legal sequence of $\Add$ operations to go from the $k$-diagram (left) to the $c^b$-diagram (right) for $\DetRecmn$.

\begin{figure}[ht]
\centering
\begin{tikzpicture}[scale = 0.55]
\ferrers{2,1,1}{(0,0)}
\draw [thick] (2,-1)--(3,-1)--(3,-2)--(2,-2)--(2,-1);
\draw [red,very thick,->] (3.7,-1.8)--(7.3,-1.8);
\node [below] at (5.5,-1.8) {\textcolor{red}{$\Add$}};
\ferrers{3,1,1}{(8,0)}
\draw [red, thick] (9,-2)--(10,-2)--(10,-3)--(9,-3)--(9,-2);
\draw [red,very thick,->] (11.7,-1.8)--(15.3,-1.8);
\node [below] at (13.5,-1.8) {\textcolor{red}{$\Add$}};
\ferrers{3,2,1}{(16,0)}
\draw [red, thick] (17,-3)--(18,-3)--(18,-4)--(17,-4)--(17,-3);
\end{tikzpicture}
\caption{Illustrating a legal sequence from $F(1,1,3)$ to $F(2,2,3)$, showing that the configuration $c = (0,2,2;2,2,3)$ is deterministically recurrent. \label{fig:comp_ASM}}
\end{figure}
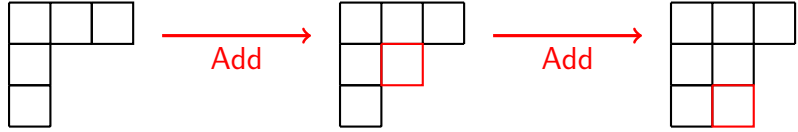
\end{example}

As in the SSM case, we can give a DAG interpretation of Theorem~\ref{thm:bij_rec_fer_ASM}. Here the vertices are all Ferrers diagrams in $\Fermn$, and we put an edge from $F$ to $F'$ if $F' = \Add(F)$. For consistency with the SSM case we colour the edges \textcolor{red}{red}. We denote $\DAGDmn$ the DAG thus obtained. $\DAGDmn$ is also \emph{bipolar} in the ASM case with the same source $F(0, \cdots, 0, m)$ and sink $F(m, \cdots, m)$. With this representation, $\Phi$ is a bijection from $\DetSortRecmn$ to pairs of vertices $(F,F')$ of $\DAGDmn$ such that there is a (directed) path from $F$ to $F'$ in $\DAGDmn$. The level of the configuration is equal to the number of (red) edges in such a path. Figure~\ref{fig:dag_ASM} illustrates this construction.

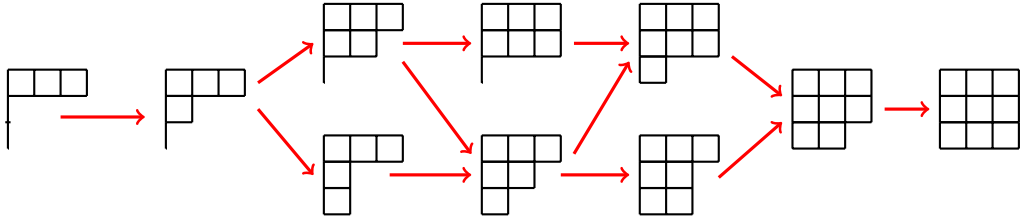
\begin{figure}[ht]
\centering
\begin{tikzpicture}[scale=0.35]
 %level 0
 \ferrers{3,0,0}{(0,-4.5)}
 \draw [thick] (-0.1,-7.5)--(0.1,-7.5);
 %level 0 to level 1
 \draw [red,very thick,->] (2,-7.3)--(5.2,-7.3);
 %level 1
 \ferrers{3,1,0}{(6,-4.5)}
 %level 1 to level 2
 \draw [red,very thick,->] (9.5,-6)--(11.6,-4.5);
 \draw [red,very thick,->] (9.5,-7)--(11.6,-9.5);
 %level 2
 \ferrers{3,2,0}{(12,-2)}
 \ferrers{3,1,1}{(12,-7)}
 %level 2 to level 3
 \draw [red,very thick,->] (15,-4.5)--(17.6,-4.5);
 \draw [red,very thick,->] (15,-5.2)--(17.6,-8.7);
 \draw [red,very thick,->] (14.5,-9.5)--(17.6,-9.5);
 %level 3
 \ferrers{3,3,0}{(18,-2)}
 \ferrers{3,2,1}{(18,-7)}
 %level 3 to level 4
 \draw [red,very thick,->] (21.5,-4.5)--(23.6,-4.5);
 \draw [red,very thick,->] (21.5,-8.7)--(23.6,-5.2);
 \draw [red,very thick,->] (21,-9.5)--(23.6,-9.5);
 %level 4
 \ferrers{3,3,1}{(24,-2)}
 \ferrers{3,2,2}{(24,-7)}
 %level 4 to level 5
 \draw [red,very thick,->] (27.5,-5)--(29.4,-6.5);
 \draw [red,very thick,->] (27,-9.6)--(29.4,-7.5);
 %level 5
 \ferrers{3,3,2}{(29.8,-4.5)}
 %level 5 to level 6
 \draw [red,very thick,->] (33.3,-7)--(35,-7);
 %level 6
 \ferrers{3,3,3}{(35.4,-4.5)}
\end{tikzpicture}
\caption{The graph $\DAGD{3}{3}$ for the ASM. \label{fig:dag_ASM}}
\end{figure}

\subsection{Connection to parallelogram polyominoes}\label{subsec:para_poly}

In this section, we use our characterisation with Ferrers diagrams to recover results from Dukes and Le Borgne~\cite{DLB} that link DR configurations on complete bipartite graphs to \emph{parallelogram polyominoes}. A parallelogram polyomino is a collection of cells in a bounding rectangle $[0,m] \times [0,n]$ which lies between two lattice paths $\U$ (the \emph{upper} path) and $\L$ (the \emph{lower} path) with the following properties:
\begin{itemize}[topsep=2pt]
\item the paths $\U$ and $\L$ go from $(0,0)$ to $(m,n)$, with steps $N = (0,1)$ and $E = (1,0)$;
\item the paths $\U$ and $\L$ do not intersect except at $(0,0)$ and $(m,n)$.
\end{itemize}
We write $P = P\left( \U, \L \right)$ for the parallelogram polyomino with upper path $\U$ and lower path $\L$. We denote $\Para{m}{n}$ the set of parallelogram polyominoes with bounding box $[0,m] \times [0,n]$. The \emph{area} of a parallelogram polyomino $P$, denoted $\mathrm{Area}(P)$, is its number of cells. Figure~\ref{fig:para_poly_ex} shows an example of a parallelogram polyomino $P \in \Para{5}{3}$ with area $10$.

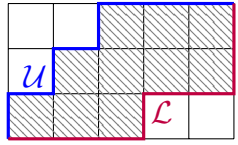
\begin{figure}[ht]
\centering
\begin{tikzpicture}[scale=0.6]
 \draw (0,0) grid (5,3);
 %filled
 \draw [pattern=north west lines, pattern color=gray] (0,0)--(3,0)--(3,1)--(5,1)--(5,3)--(2,3)--(2,2)--(1,2)--(1,1)--(0,1)--cycle;
 %upper
 \draw [blue, very thick] (0,0)--(0,1)--(1,1)--(1,2)--(2,2)--(2,3)--(5,3);
 \node [blue] at (0.6,1.4) {$\U$};
 %lower
 \draw [purple, very thick] (0,0)--(3,0)--(3,1)--(5,1)--(5,3);
 \node [purple] at (3.4,0.6) {$\L$};
 \end{tikzpicture}
 
 \caption{An example of a parallelogram polyomino $P \in \Para{5}{3}$ with upper path $\U$ (in \textcolor{blue}{blue}), lower path $\L$ (in \textcolor{purple}{purple}), and $\mathrm{Area}(P) = 10$.\label{fig:para_poly_ex}}

\end{figure}

Dukes and Le Borgne in~\cite{DLB} described a bijection from $\DetSortRecmn$ to $\Para{m+1}{n}$ as follows. Given a \emph{sorted} configuration $c = (c^t;c^b) \in \Stablemn$, they defined two lattice paths.
\begin{itemize}[topsep=2pt]
\item The upper path $\U = \U\left( c^t \right)$ is the path from $(0,0)$ to $(m+1,n)$ whose $E$ steps occur at heights $(1+c^t_1,\cdots,1+c^t_{m},n)$ above the $x$-axis. One may think of the final $E$ step as corresponding to the sink vertex $v^t_0$, with the convention $c^t_0=n-1$.
\item The lower path $\L = \L\left( c^b \right)$ is the path from $(0,0)$ to $(m+1,n)$ whose $N$ steps occur at heights $(1+c^b_1,\cdots,1+c^b_{n})$ to the right of the $y$-axis.
\end{itemize}

\begin{theorem}[{\cite[Theorem~3.7]{DLB}}]\label{thm:bij_rec_para}
For a sorted configuration $c = (c^t;c^b) \in \Stablemn$, let $\U =\U\left( c^t \right)$ and $\L = \L\left( c^b \right)$ be defined as above. 
Then the map $\Phi : c \mapsto P := P\left( \U, \L \right)$ is a bijection from $\DetSortRecmn$ to $\Para{m+1}{n}$. 
Moreover, if $c$ is DR, we have $\level{c} = \mathrm{Area}(P) - m - n$.
\end{theorem}

\begin{example}\label{ex:config_to_poly}
Consider the polyomino $P$ from Figure~\ref{fig:para_poly_ex}. It has upper path $\U = \U(P) = NENENEEE$. The heights of its $E$ steps are $(1,2,3,3,3)$, which corresponds to the top configuration $c^t = (0, 1, 2, 2)$. Similarly, the lower path $\L = \L(P) = EEENEENN$ has $N$ steps at heights $(3,5,5)$, yielding $c^b = (2, 4, 4)$. Finally, we get the configuration $c = (0, 1, 2, 2; 2, 4, 4)$, which one can check is indeed DR by Theorem~\ref{thm:charac_ASM} (here we have $k = (1, 2, 4)$). 
\end{example}

In fact, we can recover this bijection by using our characterisation as pairs of strongly compatible Ferrers diagrams. Indeed, given a pair $(F_1, F_2) \in \Fermn \times \Fermn$ of strongly compatible Ferrers diagrams, we construct the parallelogram polyomino as follows. Let $F_1' \in \Fer{m+1}{n+1}$ be the Ferrers diagram obtained from $F_1$ by adding one empty row at the bottom, and an extra cell to (the right of) the top row. Let $F_2' \in \Fer{m+1}{n}$ be the Ferrers diagram obtained from $F_2$ by adding one extra cell to each row, and one ``full'' row at the top (with $m+1$ cells). Note that the pair $(F_1', F_2')$ is strongly compatible, and that both $F_1'$ and $F_2'$ have $m+1$ cells in their top row. In particular, the difference $F_2' \setminus F_1'$ consisting of cells that are in $F_2'$ but not in $F_1'$ is contained in the box $[0,m+1] \times [0,n]$. We denote $\mathrm{Diff} \left(F_1, F_2 \right) := F_2' \setminus F_1'$ the collection of cells thus obtained. We will see that $\mathrm{Diff} \left(F_1, F_2 \right)$ turns out to always be a parallelogram polyomino.

\begin{example}\label{ex:Ferrers_ParaPoly_ASM}
Consider the pair $(F_1, F_2)$ of strongly compatible Ferrers diagrams given by $F_1 = F(1, 1, 3)$ and $F_2 = F(2, 2, 3)$ (Figure~\ref{fig:Ferrers_ParaPoly}, left). This yields $F_1' = F(0,1,1,4)$ and $F_2' = F(3, 3, 4, 4)$ (Figure~\ref{fig:Ferrers_ParaPoly}, centre). Then taking the difference $\mathrm{Diff}\left(F_1, F_2 \right)$ gives the collection of shaded cells in the right of Figure~\ref{fig:Ferrers_ParaPoly}. Note that $\mathrm{Diff}\left(F_1, F_2 \right)$ is a parallelogram polyomino, as claimed above.

\begin{figure}[ht]

\centering

\begin{tikzpicture}[scale = 0.55]
%F_1
\ferrers{3,1,1}{(0,0)}
\node at (1.5,-4.6) {$F_1$};
\draw [thick, ->] (3.5,-2.7)--(7.5,-2.3);
%F_2
\ferrers{3,2,2}{(0,-6)}
\node at (1.5,-10.6) {$F_2$};
\draw [thick, ->] (3.5,-8.2)--(7.5,-8.8);
%F_1'
\ferrers{4,1,1,0}{(8,0.5)}
\node at (10,-4.6) {$F_1'$};
\draw [blue, very thick] (12,-0.5)--(12,-1.5)--(9,-1.5)--(9,-3.5)--(8,-3.5)--(8,-4.5);
\draw [thick, ->] (12.5,-2.5)--(16.5,-5);
%F_2'
\ferrers{4,4,3,3}{(8,-5.5)}
\node at (10,-11.1) {$F_2'$};
\draw [purple, very thick] (12,-6.5)--(12,-8.5)--(11,-8.5)--(11,-10.5)--(8,-10.5);
\draw [thick, ->] (12.8,-9)--(16.5,-6);
 %Para poly
 \begin{scope}[shift={(17,-7.5)}]
 \begin{scope}[scale=1.2]
 \draw (0,0) grid (4,3);
 %filled
 \draw [pattern=north west lines, pattern color=gray] (0,0)--(3,0)--(3,2)--(4,2)--(4,3)--(1,3)--(1,1)--(0,1)--cycle;
 %upper
 \draw [blue, very thick] (0,0)--(0,1)--(1,1)--(1,3)--(4,3);
 \node [blue] at (0.6,1.4) {$\U$};
 %lower
 \draw [purple, very thick] (0,0)--(3,0)--(3,2)--(4,2)--(4,3);
 \node [purple] at (3.4,0.6) {$\L$};
 \end{scope}
 \node at (2.5,-1) {$\mathrm{Diff}\left(F_1, F_2 \right) := F_2' \setminus F_1'$};
 \end{scope}
\end{tikzpicture}

\caption{Illustrating the construction from a pair $(F_1, F_2)$ of strongly compatible Ferrers diagrams to the polyomino $\mathrm{Diff} \left( F_1, F_2 \right)$. \label{fig:Ferrers_ParaPoly}}
\end{figure}
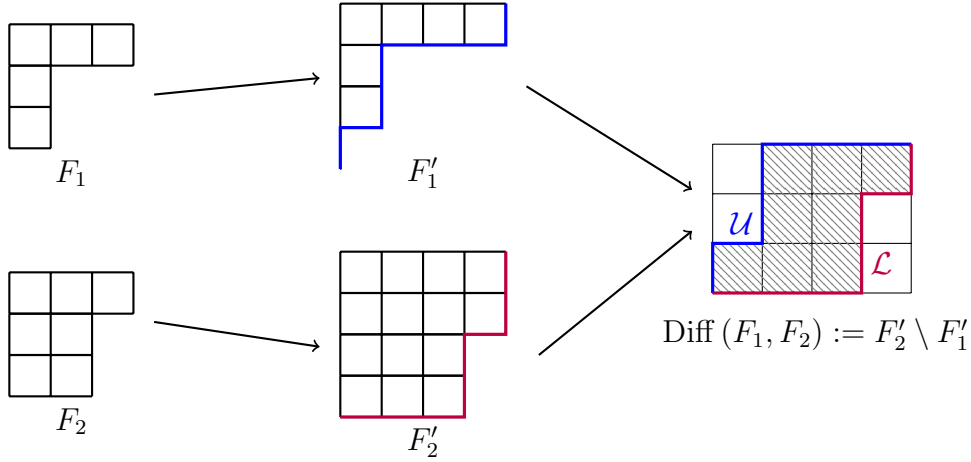

\end{example}

\begin{theorem}\label{thm:ferrer_para_poly}
We have $\Phi = \mathrm{Diff} \circ \Psi$, where $\Phi$ is the Dukes-Le Borgne bijection from sorted DR configurations to parallelogram polyominoes (Theorem~\ref{thm:bij_rec_para}) and $\Psi$ is our bijection from sorted DR configurations to pairs of strongly compatible Ferrers diagrams (Theorem~\ref{thm:bij_rec_fer_ASM}).
\end{theorem}

\begin{proof}
The result follows from the various constructions combined with Remark~\ref{rem:k-diag_to_config} which details how the top configuration $c^t$ can be read from the Ferrers diagram $F(k)$.
\end{proof}

\subsection{Connection to Motzkin paths}\label{subsec:motzkin}

Finally, in this section we exhibit a bijection between sorted DR configurations and a family of labelled \emph{Motzkin paths}. A Motzkin path is a lattice path from $(0,0)$ to $(k,0)$ for some $k \geq 0$ taking steps $U = (1,1)$ (\emph{up} step), $H = (1,0)$ (\emph{horizontal} step), and $D = (1,-1)$ (\emph{down} step), which never goes below the $x$-axis. The \emph{area} of a Motzkin path $w$, denoted by $\mathrm{Area}(w)$, is the area between its curve and the $x$-axis (see Example~\ref{exa:para_motz}).  We will consider Motzkin paths whose $H$-steps are labelled either $N$ or $E$, and we denote these steps $H^N, H^E$. There is a classical bijection between parallelogram polyominoes and these labelled Motzkin paths (see~\cite[Section~3]{DV}), which we recall briefly here.

Let $P \in \Para{m+1}{n}$ be a parallelogram polyomino, with upper and lower paths $\U = (N, u_1, \cdots, u_{m+n-1}, E)$ and $\L = (E, \ell_1, \cdots, \ell_{m+n-1}, N)$. We define a lattice path $w = w(P) = w_1 \cdots w_{m+n-1}$ by setting, for each $i \in [m+n-1]$:
\begin{equation}\label{eq:poly_to_motzkin}
w_i := \begin{cases}
 U \quad \text{if } u_i = N \text{ and } \ell_i = E, \\
 H^N \quad \text{if } u_i = N \text{ and } \ell_i = N, \\
 H^E \quad \text{if } u_i = E \text{ and } \ell_i = E, \\
 D \quad \text{if } u_i = E \text{ and } \ell_i = N. \\
\end{cases}
\end{equation}
Another way of defining the path $w$ is to consider the \emph{diagonal distance} between the two paths $\U$ and $\L$. These are the distances along the lines $\Delta_i$ with equation $y = i - x$ for $i \in [1, m+n]$. For $i \in [m+n-1]$, if the distance between $\U$ and $\L$ increases from $\Delta_i$ to $\Delta_{i+1}$, we set $w_i = U$. Similarly, if the distance decreases, we set $w_i = D$, and if the distance remains the same we set $w_i = H^*$ where $* = N$ or $E$ depending on whether the line ``slides'' upwards ($* = N$) or rightwards ($* = E$).

\begin{example}\label{exa:para_motz}
Consider the parallelogram polyomino $P$ on the left of Figure~\ref{fig:exa_poly_motzkin} below. We have $\U(P) = (N)NNEEENNEE(E)$ and $\L(P) = (E)EENENENEN(N)$, putting the first and last steps of each path in parentheses as they will be ignored in the construction of the path $w$. We have $u_1 = u_2 = N$ and $\ell_1 = \ell_2 = E$, so we put $w_1 = w_2 = U$ following Equation~\eqref{eq:poly_to_motzkin}. Then $u_3 = E$ and $\ell_3 = N$, so we put $w_3 = D$. We continue through each pair of steps of the upper and lower paths, applying Equation~\eqref{eq:poly_to_motzkin} at each step. Finally, we get $w = UUDH^EDUH^NH^ED$ (Figure~\ref{fig:exa_poly_motzkin}, right). Alternatively, we can look at the diagonal distances along the lines $\Delta_i$. We have coloured the intersection of $P$ with the region between $\Delta_i$ and $\Delta_{i+1}$ and the region below the corresponding step $w_i$ using the same colour. For example, the \textcolor{red}{red} region (corresponding to $i=7$) sees the distance between $\U$ and $\L$ stay the same, ``sliding'' upwards, so the corresponding step is $w_7 = H^N$, and the region below this step in $w$ has also been coloured \textcolor{red}{red}. The area of the Motzkin path is $\mathrm{Area}(w) = \frac12 + \frac32 + \frac32 + 1 + \frac12 + \frac12 + 1 + 1 + \frac12 = 8$.

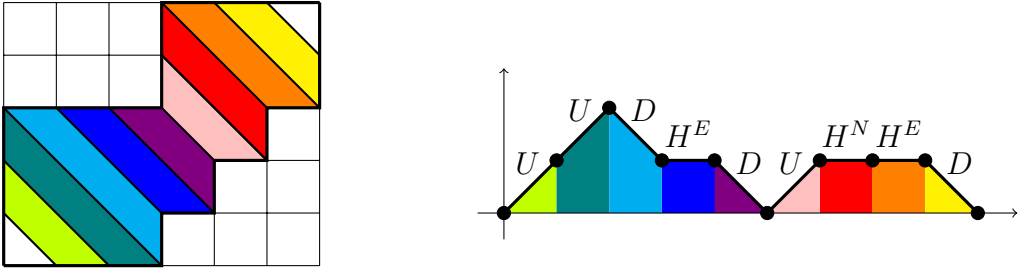
\begin{figure}[ht]
 	\centering
  \begin{tikzpicture}[scale=0.35]
   % bounding box grid
   \draw [step=2] (0,-6) grid (12,-16);
	
	% region with colors
%    \fill [color=lime!50] (0,-12)--(0,-14)--(2,-16)--(4,-16);
%    \fill [color=teal!40] (0,-10)--(0,-12)--(4,-16)--(6,-16);
%    \fill [color=cyan!20] (0,-10)--(2,-10)--(6,-14)--(6,-16);
%	\fill [color=blue!40] (2,-10)--(6,-14)--(8,-14)--(4,-10);    
%    \fill [color=violet!34] (4,-10)--(6,-10)--(8,-12)--(8,-14);
%    \fill [color=pink!50]  (6,-8)--(6,-10)--(8,-12)--(10,-12);    
%    \fill [color=red!34]  (6,-6)--(6,-8)--(10,-12)--(10,-10);   
%	\fill [color=orange!40]  (6,-6)--(8,-6)--(12,-10)--(10,-10);	
%	\fill [color=yellow!50]  (8,-6)--(10,-6)--(12,-8)--(12,-10);	
    \fill [color=lime] (0,-12)--(0,-14)--(2,-16)--(4,-16);
    \fill [color=teal] (0,-10)--(0,-12)--(4,-16)--(6,-16);
    \fill [color=cyan] (0,-10)--(2,-10)--(6,-14)--(6,-16);
	\fill [color=blue] (2,-10)--(6,-14)--(8,-14)--(4,-10);    
    \fill [color=violet] (4,-10)--(6,-10)--(8,-12)--(8,-14);
    \fill [color=pink]  (6,-8)--(6,-10)--(8,-12)--(10,-12);    
    \fill [color=red]  (6,-6)--(6,-8)--(10,-12)--(10,-10);   
	\fill [color=orange]  (6,-6)--(8,-6)--(12,-10)--(10,-10);	
	\fill [color=yellow]  (8,-6)--(10,-6)--(12,-8)--(12,-10);	
   
    % diagonal lines
    \draw [thick, color=black]  (0,-14)--(2,-16);
    \draw [thick, color=black] (0,-12)--(4,-16);
    \draw [thick, color=black] (0,-10)--(6,-16);
    \draw [thick, color=black] (2,-10)--(6,-14);
    \draw [thick, color=black] (4,-10)--(8,-14);
    \draw [thick, color=black] (6,-10)--(8,-12);
    \draw [thick, color=black]  (6,-8)--(10,-12);
    \draw [thick, color=black]  (6,-8)--(10,-12);
	\draw [thick, color=black]  (6,-6)--(10,-10);
	\draw [thick, color=black]  (8,-6)--(12,-10);
	\draw [thick, color=black]  (10,-6)--(12,-8);
   
	% lower path
   \draw [very thick]  (0,-16)--(6,-16)--(6,-14)--(8,-14)--(8,-12)--(10,-12)--(10,-10)--(12,-10)--(12,-6);   
   
   % upper path
   \draw [very thick]  (0,-16)--(0,-10)--(6,-10)--(6,-6)--(12,-6);
   
   \begin{scope}[shift={(19,-14)}]
	
	% region with colors
%    \fill [color=lime!50] (0,-0)--(2,0)--(2,2);
%    \fill [color=teal!40] (2,0)--(4,0)--(4,4)--(2,2);
%    \fill [color=cyan!20] (4,0)--(6,0)--(6,2)--(4,4);
%	\fill [color=blue!40] (6,0)--(8,0)--(8,2)--(6,2);    
%    \fill [color=violet!34] (8,0)--(10,0)--(8,2);
%    \fill [color=pink!50]  (10,0)--(12,0)--(12,2);    
%    \fill [color=red!34]  (12,0)--(14,0)--(14,2)--(12,2);   
%	\fill [color=orange!40]  (14,0)--(16,0)--(16,2)--(14,2);	
%	\fill [color=yellow!50]  (16,0)--(18,0)--(16,2);
    \fill [color=lime] (0,-0)--(2,0)--(2,2);
    \fill [color=teal] (2,0)--(4,0)--(4,4)--(2,2);
    \fill [color=cyan] (4,0)--(6,0)--(6,2)--(4,4);
	\fill [color=blue] (6,0)--(8,0)--(8,2)--(6,2);    
    \fill [color=violet] (8,0)--(10,0)--(8,2);
    \fill [color=pink]  (10,0)--(12,0)--(12,2);    
    \fill [color=red]  (12,0)--(14,0)--(14,2)--(12,2);   
	\fill [color=orange]  (14,0)--(16,0)--(16,2)--(14,2);	
	\fill [color=yellow]  (16,0)--(18,0)--(16,2);	
   
   %axes
   \draw [->] (-1,0)--(19.5,0);
   \draw [->] (0,-1)--(0,5.5);
   
   %path
   \draw [very thick] (0,0)--(4,4)--(6,2)--(8,2)--(10,0)--(12,2)--(16,2)--(18,0);
   \tdot{0}{0}{black}
   \tdot{2}{2}{black}
   \tdot{4}{4}{black}
   \tdot{6}{2}{black}
   \tdot{8}{2}{black}
   \tdot{10}{0}{black}
   \tdot{12}{2}{black}
   \tdot{14}{2}{black}
   \tdot{16}{2}{black}
   \tdot{18}{0}{black}
   
   %labels
   \node at (0.9,1.9) {$U$};
   \node at (2.9,3.9) {$U$};
   \node at (5.3,3.9) {$D$};
   \node at (7,3) {$H^E$};
   \node at (9.3,1.9) {$D$};
   \node at (10.9,1.9) {$U$};
   \node at (13,3) {$H^N$};
   \node at (15,3) {$H^E$};
   \node at (17.3,1.9) {$D$};
   
   \end{scope}
	
  \end{tikzpicture}
  \caption{Illustrating the construction from polyominoes to labelled Motzkin paths using the diagonal distance.\label{fig:exa_poly_motzkin}}
\end{figure}  

\end{example}

Note that the path $w$ thus obtained is indeed a Motzkin path with labelled $H$-steps. Moreover, given the dimension of the parallelogram polyomino ($(m+1) \times n$), $w$ has $m+n-1$ steps in total, of which exactly $m$ are either $U$ or $H^E$ steps (equivalently $n-1$ steps $D$ or $H^N$, equivalently $m$ steps $D$ or $H^E$, equivalently $n-1$ steps $U$ or $H^N$). We denote $\Motz{m}{n-1}$ the set of such labelled Motzkin paths.

\begin{theorem}[{\cite[Section~3]{DV}}]\label{thm:bij_para_motz}
The map $\Xi : \Para{m+1}{n} \rightarrow \Motz{m}{n-1}, P \mapsto w(P)$, where $w$ is defined by Equation~\eqref{eq:poly_to_motzkin}, is a bijection.
\end{theorem}

By combining this map $\Xi$ and the bijection $\Phi$ from Theorem~\ref{thm:bij_rec_para}, we can get a bijection from $\Motz{m}{n-1}$ to $\DetSortRecmn$. Let us conclude this paper by giving direct descriptions of this bijection and its converse.

\begin{definition}\label{def:motz_DR}
Let $w = w_1 \cdots w_{m+n-1} \in \Motz{m}{n-1}$. We construct a sorted configuration $c = c(w) = (c^t; c^b) \in \Configmn$ as follows.
\begin{itemize}[topsep=2pt]
\item For $i \in [m]$, let $j$ be such that $w_j$ is the $i$-th step of the form $D$ or $H^E$ in $w$. Then we set $c^t_i := \sharp \left\{ j' < j; \, w_{j'} \in \left\{U, H^N\right\} \right\}$.
\item For $j \in [n-1]$, let $i$ be such that $w_i$ is the $j$-th step of the form $D$ or $H^N$ in $w$. Then we set $c^b_j := \sharp \left\{ i' < i; \, w_{i'} \in \left\{U, H^E\right\} \right\}$. Finally, we also set $c^b_n := m$.
\end{itemize}
\end{definition}

\begin{example}
Consider the labelled Motzkin path from the right of Figure~\ref{fig:exa_poly_motzkin}, given by $w = UUDH^EDUH^NH^ED$. First let us calculate $c^t$. For $i=1$, the first step of the form $D$ or $H^E$ is $w_3 = D$. There are $2$ steps before it of the form $U$ or $H^N$ (given by $w_1 = w_2 = U$), so we set $c^t_1 = 2$. Next, for $i=2, 3$, we see that the second and third steps in $w$ of the form $D$ or $H^E$ are $w_4 = H^E$ and $w_5 = D$, both of which also have the same $2$ steps before them of the form $U$ or $H^N$. We therefore set $c^t_2 = c^t_3 = 2$. Finally, for $i = 4, 5$, we see that the fourth and fifth steps in $w$ of the form $D$ or $H^E$ are $w_8 = H^E$ and $w_9 = D$, both of which have $4$ steps before them of the form $U$ or $H^N$ (given by $w_1 = U$, $w_2 = U$, $w_6 = U$ and $w_7 = H^N$). We therefore set $c^t_4 = c^t_5 = 4$, yielding $c^t = (2, 2, 2, 4, 4)$. The same process yields the bottom configuration $c^b = (2,3,4,5,5)$, where the last value is set according to the rule $c^b_n = m$. Finally, we get the configuration $c = (2,2,2,4,4; 2,3,4,5,5)$, and we can check that this is the same configuration corresponding to the parallelogram polyomino from the left of Figure~\ref{fig:exa_poly_motzkin} according to the Dukes-Le Borgne bijection of Theorem~\ref{thm:bij_rec_para}.
\end{example}

Algorithm~\ref{algo:motz_DR} provides a construction of $c = c(w)$ that uses a single loop over the path $w$.

\begin{algorithm}[ht]
\caption{The bijection from $\Motz{m}{n-1}$ to $\DetSortRecmn$}\label{algo:motz_DR}
  \begin{algorithmic}
      \Require $w = w_1 \cdots w_{m+n-1} \in \Motz{m}{n-1}$
      \State \textbf{Initialise:} $c^t = c^b = ()$ (empty lists); $\mathrm{tVal} = \mathrm{bVal} = 0$; $i=j=1$
      \For {$k$ from $1$ to $m+n-1$} 
        \If {$w_k = U$}
          \State $\mathrm{tVal} \gets \mathrm{tVal + 1}$; $\mathrm{bVal} \gets \mathrm{bVal + 1}$
        \ElsIf {$w_k = H^E$}
          \State $c^t.\mathtt{append}(\mathrm{tVal})$; $\mathrm{bVal} \gets \mathrm{bVal + 1}$
        \ElsIf {$w_k = H^N$}
          \State $c^b.\mathtt{append}(\mathrm{bVal})$; $\mathrm{tVal} \gets \mathrm{tVal + 1}$
        \Else \Comment{$w_k = D$}
          \State $c^t.\mathtt{append}(\mathrm{tVal})$; $c^b.\mathtt{append}(\mathrm{bVal})$
        \EndIf
      \EndFor
      \State $c^b.\mathtt{append}(m)$ \Comment{to get $c^b_n = m$} \\
      \Return $c = (c^t; c^b) \in \DetSortRecmn$
  \end{algorithmic}
\end{algorithm}

\begin{theorem}\label{thm:bij_motz_DR}
The map $w \mapsto c(w)$ described by Algorithm~\ref{algo:motz_DR} is the bijection $\Phi^{-1} \circ \Xi^{-1} : \Motz{m}{n-1} \rightarrow \DetSortRecmn$.
\end{theorem}

\begin{proof}
It is straightforward to see that Algorithm~\ref{algo:motz_DR} outputs the configuration $c = c(w) = (c^t_1, \cdots, c^t_m; c^b_1, \cdots c^b_n)$ given in Definition~\ref{def:motz_DR}. The result then follows from Theorem~\ref{thm:bij_rec_para}, Equation~\eqref{eq:poly_to_motzkin} and Theorem~\ref{thm:bij_para_motz}.
\end{proof}

Finally, Algorithm~\ref{algo:DR_motz} shows the converse, i.e.\ how to construct the path $w = w(c) \in \Motz{m}{n-1}$ from a sorted DR configuration $c = (c^t; c^b)$. To do this, we put the sorted top and bottom configurations $c^t$ and $c^b$ into \emph{sorted stacks}, with the heads of each stack being its smallest element. In the stack $c^t$, we place an additional element $c^t_{m+1} = n-1$ at the tail. The operation $\mathtt{head}()$ accesses (returns) the element at the head of the stack, while $\mathtt{pop}()$ removes it. Finally, we denote $\mathbf{1}$ the vector whose entries are all $1$'s.

\begin{algorithm}[ht]\caption{The bijection from $\DetSortRecmn$ to $\Motz{m}{n-1}$}\label{algo:DR_motz}
  \begin{algorithmic}[1]
    \Require sorted stacks $c^t = (c^t_1 \leq  \cdots \leq c^t_m \leq c^t_{m+1} = n-1)$, $c^b = (c^b_1 \leq \cdots \leq c^b_n)$
    \State \textbf{Initialise:} empty list $w = ()$
    \While{$c^t \neq \emptyset \And c^b \neq \emptyset$}
      \While{$c^t.\mathrm{head}() > 0 \And c^b.\mathrm{head}() > 0$}
        \State $w.\texttt{append}(U)$
        \State $c^t \gets c^t - \mathbf{1}$; $c^b \gets c^b - \mathbf{1}$ \Comment{Decrease all values in $c^t$ and $c^b$ by one}
      \EndWhile
      \If{$0 = c^t.\mathrm{head}() < c^b.\mathrm{head}()$}
        \State $c^t.\mathtt{pop}()$; $w.\texttt{append}(H^E)$
        \State $c^b \gets c^b - \mathbf{1}$ 
      \ElsIf{$0 = c^b.\mathrm{head}() < c^t.\mathrm{head}()$}
        \State $c^b.\mathtt{pop}()$; $w.\texttt{append}(H^N)$
        \State $c^t \gets c^t - \mathbf{1}$ 
      \Else \Comment{$c^b.\mathrm{head}() = c^t.\mathrm{head}() = 0$}
        \State $c^t.\mathtt{pop}()$; $c^b.\mathtt{pop}()$; $w.\texttt{append}(D)$
      \EndIf
    \EndWhile
    \State \textbf{Delete} final $D$-step from $w = (w_1, \cdots, w_{m+n-1}, w_{m+n} = D)$ \\
    \Return $w = (w_1, \cdots, w_{m+n-1}) \in \Motz{m}{n-1}$
  \end{algorithmic}
\end{algorithm}

\begin{example}\label{exa:bij_DR_motz}
Consider the sorted DR configuration $c = (2,2,2,4,4; 2,3,4,5,5)$ corresponding to the parallelogram polyomino $P$ on the left of Figure~\ref{fig:exa_poly_motzkin}. We first put an additional element into the top stack $c^t$, yielding the starting stacks $c^t = (2, 2, 2, 4, 4, 4)$ and $c^b = (2, 3, 4, 5, 5)$.
\begin{itemize}
\item Initially we have $c^t_1 = c^b_1 = 2$. This means that we apply the steps of lines~4 and 5 twice, yielding $w = UU$, $c^t = (0,0,0,2,2,2)$, and $c^b = (0,1,2,3,3)$.
\item Here we are in the last case, where $c^b.\mathrm{head}() = c^t.\mathrm{head}() = 0$, so we apply line~14 of the algorithm. This gives $w = UUD$, $c^t = (0,0,2,2,2)$, and $c^b = (1,2,3,3)$.
\item We are now in the first case, where $0 = c^t.\mathrm{head}() < c^b.\mathrm{head}()$, so we apply lines~8 and 9 of the algorithm, yielding $w = UUDH^E$, $c^t = (0,2,2,2)$, and $c^b = (0,1,2,2)$.
\item We are once again in the last case, so we apply line~14, yielding $w = UUDH^ED$, $c^t = (2,2,2)$, and $c^b = (1,2,2)$.
\item Now we have $c^t.\mathrm{head}() = 2$ and $c^b.\mathrm{head}() = 1$, so we apply lines~4 and 5 once, yielding $w = UUDH^EDU$, $c^t = (1,1,1)$, and $c^b = (0,1,1)$.
\item Here we are in the second case, where $0 = c^b.\mathrm{head}() < c^t.\mathrm{head}()$, so we apply lines~11 and 12 of the algorithm, yielding $w = UUDH^EDUH^N$, $c^t = (0,0,0)$, and $c^b = (1,1)$.
\item Now we are in the first case, where $0 = c^t.\mathrm{head}() < c^b.\mathrm{head}()$, so we apply lines~8 and 9 of the algorithm, yielding $w = UUDH^EDUH^NH^E$, $c^t = (0,0)$, and $c^b = (0,0)$.
\item We will then apply line~14 twice more to exit the outer while loop, emptying both stacks, and yielding $w = UUDH^EDUH^NH^EDD$.
\item Finally, we delete the final $D$ step from $w$, so that the algorithm returns $w = UUDH^EDUH^NH^ED$. We note that this is indeed the labelled Motzkin path from the right of Figure~\ref{fig:exa_poly_motzkin}.
\end{itemize}
\end{example}

\begin{theorem}\label{thm:bij_DR_motz}
The map $c \mapsto w(c)$ described by Algorithm~\ref{algo:DR_motz} is the bijection $\Xi \circ \Phi : \DetSortRecmn \rightarrow \Motz{m}{n-1}$. Moreover, for $c \in \DetSortRecmn$, we have $\level{c} = \mathrm{Area}(w(c))$.
\end{theorem}

\begin{proof}
First, note that Algorithm~\ref{algo:DR_motz} always terminates for any sorted input stacks $c^t, c^b$, and outputs a finite word $w \in \left\{U, D, H^N, H^E \right\}^*$. We may therefore define the map $A : \DetSortRecmn \rightarrow \left\{U, D, H^N, H^E \right\}^*, c \mapsto w(c)$, given by the algorithm when the input corresponds to a sorted DR configuration on $K^0_{m,n}$. We wish to show that, for any $m \geq 0$ and $n \geq 1$, we have:
\begin{equation}\label{eq:thm_bij_DR_motz}
\forall c \in \DetSortRecmn, \, A(c) = \Xi \circ \Phi (c),
\end{equation}
where $\Xi$ and $\Phi$ are the bijections from Theorems~\ref{thm:bij_para_motz} and \ref{thm:bij_rec_para} respectively.

We proceed by induction on $m, n$. The base case is $m=0, n=1$, for which $K^0_{m,n}$ is the graph consisting of a single edge from the sink $v^t_0$ to the bottom vertex $v^b_1$. In this case there is a unique DR configuration $c = (\emptyset; 0)$ (which is the only stable configuration). For the input to Algorithm~\ref{algo:DR_motz}, we set the convention $c^t_1 = n-1 = 0$, leaving the stacks $c^t = c^b = (0)$. When running the algorithm, the inner while loop (lines~3 to 6) is empty, and we are in the case of the else branch (line~13). This yields $w = D$, and empties both stacks, so that we exit the outer while loop (lines~2 to 16). The $D$ step is therefore immediately deleted, so that the algorithm outputs the empty word $w = \emptyset$. Now we also have $\Phi(c) = P \in \Para{1}{1}$, where $P$ is the parallelogram polyomino consisting of a single cell, and clearly $\Xi(P) = \emptyset$ as desired.

For the induction step, fix $m \geq 0$ and $n \geq 1$ such that $m + n > 1$. Suppose that Equation~\eqref{eq:thm_bij_DR_motz} holds for all $m' \leq m$ and $n' \leq n$ such that $m' + n' < m + n$. Let $c \in \DetSortRecmn$, with the corresponding stacks $c^t$ and $c^b$, and define $k := \min \{c^t_1, c^b_1\}$. 
We let $P := \Phi(c) \in \Para{m+1}{n}$ be the parallelogram polyomino corresponding to $c$ via the Dukes-Le Borgne bijection (Theorem~\ref{thm:bij_rec_para}), and write $\U = \U(P) = (N,u_1, \cdots, u_{m+n-1}, E)$ and $\L = \L(P) = (E, \ell_1, \cdots, \ell_{m+n-1}, N)$ for its upper and lower paths. 
Finally, we let $w = A(c)$ be the corresponding word output by Algorithm~\ref{algo:DR_motz}. We wish to show that $w = \Xi(P)$. By definition of $k$, the word $w$ starts with exactly $k$ steps $U$. Moreover, by construction of $\Phi$, we have $u_1 = \cdots = u_k = N$ and $\ell_1 = \cdots = \ell_k = E$. We distinguish the three cases defining $w_{k+1}$ given by the algorithm.

\medskip

\textbf{Case~1:} $k = c^t_1 < c^b_1$.  

\noindent By definition, we have $w_{k+1} = H^E$. Moreover, the bijection $\Phi$ gives $u_{k+1} = \ell_{k+1} = E$. Next, we define $c' := (c^t_2, \cdots, c^t_m; c^b - \mathbf{1}) \in \Config{m-1}{n}$. 
We also define $\U'$ and $\L'$ to be the paths $\U$ and $\L$ with their $(k+1)$-th step (which is $E$ for both paths) deleted, i.e.\ $\U':= (N, u_1, \cdots, u_k, u_{k+2}, \cdots, u_{m+n-1}, E)$ and $\L' := (E, \ell_1, \cdots, \ell_k, \ell_{k+2}, \cdots, \ell_{m+n-1}, N)$.
It is straighforward to check that these are the upper and lower paths of a new parallelogram polyomino $P' \in \Para{m}{n}$, which satisfies $P' = \Phi(c')$. In particular, this implies that $c' \in \DetSortRec{m-1}{n}$, and the induction hypothesis then gives $w' := A(c') = \Xi(P') \in \Motz{m-1}{n-1}$. 

Now by definition of $c'$, the transformation $c \leadsto c'$ corresponds exactly to the first modification of the input stacks $c^t, c^b$ in lines~8 and 9 of Algorithm~\ref{algo:DR_motz}, which is applied after the $k$ iterations of the inner while loop (lines~4 and 5). In other words, if $w' = A(c') = w'_1 \cdots w'_{m+n-2}$ is output by the algorithm with input $c'$, then we have $w = A(c) = w'_1 \cdots w'_k H^E w'_{k+1} \cdots w'_{m+n-2}$. The definition of $\Xi$ and construction $P \leadsto P'$, combined with the induction hypothesis, then immediately imply that $w = \Xi(P)$, as desired.

\medskip

\textbf{Case~2:} $k = c^b_1 < c^t_1$.  

\noindent This case is entirely analogous to Case~1, exchanging the roles of $c^t$ and $c^b$ and changing the steps $E$ and $H^E$ to $N$ and $H^N$ respectively.

\medskip

\textbf{Case~3:} $c^b_1 = c^t_1 = k$.

\noindent By definition, we have $w_{k+1} = D$, and the bijection $\Phi$ gives $u_{k+1} = E$ and $\ell_{k+1} = N$. First note that in this case we must have $k \geq 1$ since $c$ is DR. Indeed, otherwise we would have $\U(P) = (N, E, \cdots)$ and $\L(P) = (E, N, \cdots)$, contradicting the definition of a parallelogram polyomino. We then define $c' := (c^t_2 - 1, \cdots, c^t_m - 1; c^b_2 - 1, \cdots, c^b_n - 1) \in \Config{m-1}{n-1}$. Similarly to Case~1, we consider the corresponding parallelogram polyomino $P' = \Phi(c')$, which has upper and lower paths $\U':= (N, u_1, \cdots, u_k, u_{k+2}, \cdots, u_{m+n-1}, E)$ and $\L' := (E, \ell_1, \cdots, \ell_k, \ell_{k+2}, \cdots, \ell_{m+n-1}, N)$.

We apply the induction hypothesis to $c'$, which implies that $w' := A(c') = \Xi(P') \in \Motz{m-1}{n-2}$. Now note that, in this case, the transformation $c \leadsto c'$ corresponds exactly to initially applying the steps of the inner loop (lines~4 and 5) of Algorithm~\ref{algo:DR_motz} to the input stacks $c^t, c^b$ once, together with the first modification of the input stacks in line~14 (after all $k$ iterations of the inner loop). In other words, if $w' = A(c') = w'_1 \cdots w'_{m+n-3}$ is output by the algorithm with input $c'$, then we have $w = A(c) = U w'_1 \cdots w'_{k-1} D w'_{k} \cdots w'_{m+n-3}$. The result then follows as in Case~1 by applying the induction hypothesis to $w'$ and using the definition of the bijection $\Xi$, together with the construction $P \leadsto P'$. 

For the level formula, note that in the bijection $\Xi$ from $\Para{m+1}{n}$ to $\Motz{m}{n-1}$, each step in the Motzkin path corresponds to a ``diagonal band'' in the parallelogram polyomino (see Figure~\ref{fig:exa_poly_motzkin} for an illustration of these). It is straightforward to see that the area of each diagonal band is exactly one more than the area under the corresponding step. The result then follows from the level formula in Theorem~\ref{thm:bij_rec_para}.
\end{proof}

\section*{Acknowledgments}

The authors have no competing interests to declare that are relevant to the content of this article. The research leading to these results is partially supported by the National Natural Science Foundation of China, grant number 12101505, by the Research Development Fund of Xi'an Jiaotong-Liverpool University, grant number RDF-22-01-089, and by the Postgraduate Research Scholarship of Xi'an Jiaotong-Liverpool University, grant number PGRS2012026.

%%%%%%%%%% Bibliography
\bibliographystyle{abbrv}
\bibliography{SSM_CompBip_ArXiv_bibliography}

\begin{thebibliography}{10}

\bibitem{AADHB}
J.-C. Aval, M.~D'Adderio, M.~Dukes, A.~Hicks, and Y.~Le~Borgne.
\newblock Statistics on parallelogram polyominoes and a q,t-analogue of the
  {Narayana} numbers.
\newblock {\em Journal of Combinatorial Theory, Series A}, 123(1):271--286,
  2014.

\bibitem{AADB}
J.-C. Aval, M.~d'Adderio, M.~Dukes, and Y.~Le~Borgne.
\newblock Two operators on sandpile configurations, the sandpile model on the
  complete bipartite graph, and a cyclic lemma.
\newblock {\em Advances in Applied Mathematics}, 73:59--98, 2016.

\bibitem{BTW1}
P.~{Bak}, C.~{Tang}, and K.~{Wiesenfeld}.
\newblock {Self-organized criticality: An explanation of the 1/f noise}.
\newblock {\em {Phys. Rev. Lett.}}, 59:381--384, 1987.

\bibitem{BTW2}
P.~{Bak}, C.~{Tang}, and K.~{Wiesenfeld}.
\newblock {Self-organized criticality}.
\newblock {\em {Phys. Rev. A (3)}}, 38(1):364--374, 1988.

\bibitem{CMS}
Y.~Chan, J.-F. Marckert, and T.~Selig.
\newblock A natural stochastic extension of the sandpile model on a graph.
\newblock {\em {J. Comb. Theory, Ser. A}}, 120(7):1913--1928, 2013.

\bibitem{CLB_Baker}
R.~Cori and Y.~Le~Borgne.
\newblock On computation of {Baker} and {Norine}'s rank on complete graphs.
\newblock {\em Electron. J. Comb.}, 23(1):research paper P1.31, 47pp., 2016.

\bibitem{CorPou}
R.~{Cori} and D.~{Poulalhon}.
\newblock {Enumeration of \((p,q)\)-parking functions}.
\newblock {\em {Discrete Math.}}, 256(3):609--623, 2002.

\bibitem{CR}
R.~{Cori} and D.~{Rossin}.
\newblock {On the sandpile group of dual graphs}.
\newblock {\em {Eur. J. Comb.}}, 21(4):447--459, 2000.

\bibitem{ADILBW}
M.~D'Adderio, M.~Dukes, A.~Iraci, A.~Lazar, Y.~L. Borgne, and A.~V. Wyngaerd.
\newblock Shuffle theorems and sandpiles.
\newblock {\em ArXiv preprint}, arXiv:2401.06488, 2024.

\bibitem{DV}
M.-P. Delest and G.~Viennot.
\newblock Algebraic languages and polyominoes enumeration.
\newblock {\em Theor. Comput. Sci.}, 34:169--206, 1984.

\bibitem{DDLB}
H.~Derycke, M.~Dukes, and Y.~Le~Borgne.
\newblock The sandpile model on the complete split graph: $ q, t $-schr{\"o}der
  polynomials, sawtooth polyominoes, and a cyclic lemma.
\newblock {\em ArXiv preprint}, arXiv:2402.15372, 2024.

\bibitem{Dhar1}
D.~{Dhar}.
\newblock {Self-organized critical state of sandpile automaton models}.
\newblock {\em {Phys. Rev. Lett.}}, 64(14):1613--1616, 1990.

\bibitem{Dhar}
D.~Dhar.
\newblock Theoretical studies of self-organized criticality.
\newblock {\em Phys. A Stat. Mech. Appl.}, 369(1):29--70, 2006.
\newblock Fundamental Problems in Statistical Physics.

\bibitem{Duk}
M.~{Dukes}.
\newblock {The sandpile model on the complete split graph, Motzkin words, and
  tiered parking functions}.
\newblock {\em {J. Comb. Theory, Ser. A}}, 180:15, 2021.
\newblock Id/No 105418.

\bibitem{DLB}
M.~{Dukes} and Y.~{Le Borgne}.
\newblock {Parallelogram polyominoes, the sandpile model on a complete
  bipartite graph, and a \(q,t\)-Narayana polynomial}.
\newblock {\em {J. Comb. Theory, Ser. A}}, 120(4):816--842, 2013.

\bibitem{DSSS2}
M.~{Dukes}, T.~{Selig}, J.~P. {Smith}, and E.~{Steingr\'{\i}msson}.
\newblock {Permutation graphs and the abelian sandpile model, tiered trees and
  non-ambiguous binary trees}.
\newblock {\em {Electron. J. Comb.}}, 26(3):research paper P3.29, 25pp., 2019.

\bibitem{DSSS1}
M.~{Dukes}, T.~{Selig}, J.~P. {Smith}, and E.~{Steingr\'{\i}msson}.
\newblock {The abelian sandpile model on Ferrers graphs -- a classification of
  recurrent configurations}.
\newblock {\em {Eur. J. Comb.}}, 81:221--241, 2019.

\bibitem{SelWheel}
T.~Selig.
\newblock Combinatorial aspects of sandpile models on wheel and fan graphs.
\newblock {\em Eur. J. Comb.}, 110:23, 2023.
\newblock Id/No 103663.

\bibitem{SelSSM}
T.~{Selig}.
\newblock {The stochastic sandpile model on complete graphs}.
\newblock {\em {Electron. J. Comb.}}, 31(3):research paper P3.26, 29pp., 2024.

\bibitem{SSS}
T.~{Selig}, J.~P. {Smith}, and E.~{Steingr\'{\i}msson}.
\newblock {EW-tableaux, Le-tableaux, tree-like tableaux and the abelian
  sandpile model}.
\newblock {\em {Electron. J. Comb.}}, 25(3):research paper P3.14, 32pp., 2018.

\end{thebibliography}

\end{document}